\documentclass{amsart}

\usepackage{amsfonts,amsthm,latexsym,amsmath,amssymb,amscd,amsmath, mathrsfs, epsf,bm}
\usepackage[utf8]{inputenc}

\usepackage[dvipsnames]{xcolor}
\usepackage{tikz}
\usetikzlibrary{arrows,matrix}
\usepackage{blkarray}
\usepackage[linewidth=1pt]{mdframed}
\usepackage{mathtools}
\usepackage{longtable}
\usepackage{comment}
\usepackage{enumitem}

\usepackage{hyperref}
\hypersetup{
    colorlinks,
    linkcolor={MidnightBlue},
    citecolor={Green}
}

\usepackage[section]{algorithm}
\usepackage[nameinlink,noabbrev]{cleveref}

\makeatletter
 \newcommand{\iddef}[1]{\Hy@raisedlink{\hypertarget{#1}{}}#1}
\makeatother
\newcommand{\idref}[1]{\hyperlink{#1}{#1}}

\crefalias{enumi}{lemma}

\Crefname{algorithm}{Algorithm}{Algorithms}
\numberwithin{equation}{section}

\usepackage[margin=1.in]{geometry}
\usepackage{marginnote}
\usepackage[parfill]{parskip}
\begingroup
\makeatletter
   \@for\theoremstyle:=definition,remark,plain\do{%
     \expandafter\g@addto@macro\csname th@\theoremstyle\endcsname{%
        \addtolength\thm@preskip\parskip
     }%
   }
\endgroup

\newtheoremstyle{indented}{5pt}{5pt}{\itshape}{2.5em}{\bfseries}{.}{.5em}{}

\theoremstyle{plain}
\newtheorem{theorem}{Theorem}[section]
\newtheorem{lemma}[theorem]{Lemma}
\newtheorem{proposition}[theorem]{Proposition}
\newtheorem{corollary}[theorem]{Corollary}
\newtheorem{definition}[theorem]{Definition}

\theoremstyle{definition}
\newtheorem{example}[theorem]{Example}
\newtheorem{remark}[theorem]{Remark}

\newtheorem*{notation*}{Notation}

\theoremstyle{indented}
\newtheorem*{claim*}{\indent Claim}

\newcommand{\flatt}{\mathrm{flatt}}

\renewcommand{\bf}{\mathbf}

\newcommand{\tn}{\textnormal}
\newcommand{\im}{\tn{Im}\hspace{0.05cm}}

\newcommand{\rk}{\tn{rk}}

\newcommand{\CC}{\mathbb{C}}
\newcommand{\PP}{\mathbb{P}}

\newcommand{\bfa}{\mathbf{a}}
\newcommand{\bfb}{\mathbf{b}}
\newcommand{\bfc}{\mathbf{c}}

\newcommand{\bfe}{\mathbf{e}}

\newcommand{\bfm}{\mathbf{m}}
\newcommand{\bfn}{\mathbf{n}}

\newcommand{\bfp}{\mathbf{p}}

\newcommand{\bfu}{\mathbf{u}}
\newcommand{\bfv}{\mathbf{v}}
\newcommand{\bfw}{\mathbf{w}}
\newcommand{\bfx}{\mathbf{x}}
\newcommand{\bfy}{\mathbf{y}}

\newcommand{\calC}{\mathcal{C}}
\newcommand{\calD}{\mathcal{D}}

\newcommand{\calF}{\mathcal{F}}

\newcommand{\calO}{\mathcal{O}}

\newcommand{\calX}{\mathcal{X}}

\newcommand{\bbC}{\mathbb{C}}

\newcommand{\bbN}{\mathbb{N}}

\newcommand{\bbP}{\mathbb{P}}

\newcommand{\bbU}{\mathbb{U}}
\newcommand{\bbV}{\mathbb{V}}

\newcommand{\GL}{\mathrm{GL}}

\newcommand{\Sub}{\mathrm{Sub}}

\newcommand{\Mat}{\mathrm{Mat}}

\newcommand{\brk}{\mathrm{brk}}

\title{Decomposition loci of tensors}
\author{Alessandra Bernardi, Alessandro Oneto, Pierpaola Santarsiero}

\address[A. Bernardi, A. Oneto]{Universit\`a di Trento, Dipartimento di Matematica, Via Sommarive, 14 - 38123 Povo (Trento), Italy}
\email{alessandra.bernardi@unitn.it, alessandro.oneto@unitn.it}

\address[P. Santarsiero]{Universit\`a di Bologna, Dipartimento di Matematica, Piazza di Porta S.Donato 5, Bologna, Italy}
\email{pierpaola.santarsiero@unibo.it}

\begin{document}

\maketitle
\begin{abstract}
    The decomposition locus of a tensor is the set of rank-one tensors appearing in a minimal tensor-rank decomposition of the tensor. For tensors lying on the tangential variety of any Segre variety, but not on the variety itself, we show that the decomposition locus consists of all rank-one tensors except the tangency point only. We also explicitly compute decomposition loci of all tensors belonging to tensor spaces with finitely many orbits with respect to the action of product of general linear groups.

\end{abstract}

\section{Introduction}\label{sec:intro}
Decompositions of tensors as sums of decomposable tensors, also known as \textit{tensor-rank decompositions} or \textit{canonical polyadic decompositions}, are ubiquitous in many areas of applied mathematics and engineering, see \cite{kolda2009tensor,L,sidiropoulos2017tensor, QuantumLectures}. We are interested in decompositions having minimal length. 
\begin{definition}\label{rank1}
The \textbf{rank} of a tensor $T\in \mathbb{C}^{n_1} \otimes \cdots \otimes \mathbb{C}^{n_k}$ is the minimum integer $r$ such that 
\[
    T= \sum_{i=1}^r P_i, \text{ with } P_i = \bfv_i^1\otimes \cdots \otimes \bfv_i^k\in \mathbb{C}^{n_1} \otimes \cdots \otimes \mathbb{C}^{n_k}.
\]
We denote it by $\rk(T).$
\end{definition}
The computation of the rank is known to be an NP-hard problem \cite{hillar2013most}.

In the recent paper \cite{fawzi2022discovering}, the authors proposed a reinforcement learning approach to find minimal decomposition of the third-order tensor representing matrix multiplication: one of the most challenging problems on tensor decompositions, see \cite{landsberg2017geometry}. The authors formulate the search as a game, called \textit{TensorGame}, which goes as follows: 
\begin{itemize}
    \item the start position is a given tensor $T_0 = T$;
    \item at each step $t$, the player takes a tuple of vectors $(\bfv^1,\ldots,\bfv^k)$ and the current state by defining \[T_t \leftarrow T_{t-1} - \bfv^1\otimes\cdots\otimes\bfv^k;\]
    \item the game ends when $T_R = 0$. 
\end{itemize}
The challenge is to find, at each step, a rank-one tensor that drops the rank in such way the number of steps performed during the game are exactly the rank of the tensor $T$. 

The natural question is then the following:
\begin{center}\label{Q1}
    \textit{given a tensor $T$ what are the rank-one tensors that drop the rank of $T$?}
\end{center}
\begin{definition}\label{def:decomposition}
The \textbf{decomposition locus} of a tensor $T\in \mathbb{C}^{n_1}\otimes \cdots \otimes \mathbb{C}^{n_k}$ is
$$ 
\mathcal{D}_T := \left\{ (\bfv^1,\ldots,\bfv^k) \in \mathbb{C}^{n_1}\times \cdots \times \mathbb{C}^{n_k} ~:~ \rk(T - \lambda\bfv^1\otimes\cdots\otimes\bfv^k) = \rk(T) - 1, \text{ for some } \lambda \in \bbC\right\}.
$$
The \textbf{forbidden locus} $\mathcal{F}_T$ of $T$ is the complement of $\mathcal{D}_T$ in $\bbC^{n_1}\times\cdots\times \bbC^{n_k}$, i.e.
$$
\mathcal{F}_T: = \left\{ (\bfv^1,\ldots,\bfv^k) \in \mathbb{C}^{n_1}\times \cdots \times \mathbb{C}^{n_k} ~:~ \rk(T - \lambda \bfv^1\otimes\cdots\otimes\bfv^k) \geq \rk(T), \text{ for any } \lambda \in \bbC\right\}.
$$
\end{definition}
These objects were introduced in \cite{CCO} for symmetric tensor decompositions, also known as Waring decompositions of homogeneous polynomials, and computed for particular families of polynomials. Those ideas were used in \cite{MO} to classify homogeneous polynomials of rank at most $5$ in terms of algebraic properties of sets of points supporting their minimal decompositions. A first example of symmetric tensor having {empty} (symmetric) forbidden locus is provided in \cite{flavi2024decompositions}. A version of forbidden loci for non-reduced zero-dimensional schemes apolar to homogeneous polynomials was considered in \cite{ventura2019real}. The notion of decomposition locus is strictly related to the more geometric notion of \textit{entry locus} (see \cite{russo2009varieties,ionescu2008varieties}). More recently, in \cite{horobet2023does}, it was studied when some of the critical rank-one approximations belong to the decomposition locus of the tensor. 

In this work, we concentrate on tensors of order three for which the orbits under the action of $\GL_{n_1} \mathbb{C}\times \GL_{n_2}\mathbb{C} \times \GL_{n_3}\mathbb{C}$ are finite in number. Consequently, it should be noted that our work cannot be directly applied to compute the rank of the matrix multiplication tensor $M_{\langle 3 \rangle}$; however, we hope that it will provide valuable insights.

\subsection{Structure of the paper}
In this paper, we address the question for some special families of tensors. For sake of completeness, we first recall the well-understood case of matrices, see \Cref{sec:matrices}. In \Cref{sec:tangential}, we consider tensors that lie on the tangential variety of Segre varieties of any order. In the literature of quantum information, these are known as {\it $W$-states}, see e.g. \cite{ballico2019partially} and the references therein. Our proof of \Cref{thm:tangential} gives also a recipe to construct minimal rank decompositions of $W$-states starting from almost any rank-one tensor, see \Cref{prop:tangent_alldifferent}. In \Cref{sec:22n,sec:23n}, we consider tensors in $\bbC^2 \otimes \bbC^2 \otimes \bbC^n$ and $\bbC^2 \otimes \bbC^3 \otimes \bbC^n$, namely pencils of $2\times n$ and $3 \times n$ matrices. These are the only spaces of tensors, beside matrices, in which we have finitely many orbits with respect to the action of the product of general linear groups, see \Cref{ssec:finite_orbits}. Our computation will be explicitly made on their normal forms by using the algebra software Macaulay2 \cite{M2} and can be found at \url{https://sites.google.com/view/alessandrooneto/research/publications}. 

\subsection{Acknowledgements} 
We thank Edoardo Ballico for fruitful discussions and suggestions. The authors are members of GNSAGA (INdAM). All authors have been funded by the European Union under NextGenerationEU. PRIN 2022 Prot. n. 2022ZRRL4C\_004, Prot. n. 20223B5S8L and Prot. n. 2022E2Z4AK, respectively. Views and opinions expressed are however those of the authors only and do not necessarily reflect those of the European Union or European Commission. Neither the European Union nor the granting authority can be held responsible for them.

\begin{center}
    \includegraphics[width=\textwidth]{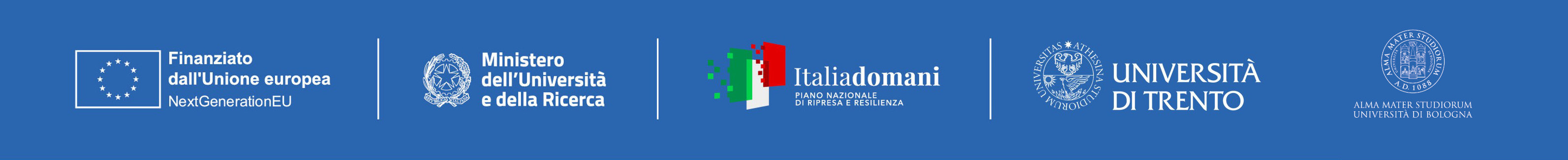}
\end{center}

\section{Preliminaries}\label{sec:preliminaries}
In this section we briefly recall some basic definition and notation. For an extensive introduction to tensor decomposition and its algebro-geometric approach see \cite{L,bernardi2018hitchhiker}.

\subsection{Segre varieties and secant varieties} 
Vectors are denoted in bold face, e.g., $\bfv$, and their corresponding projective class is denoted within brackets, e.g., $[\bfv]$. First, we recall notation for the space of rank-one tensors, i.e. \textit{Segre varieties}, and the Zariski closure of the space of rank-$r$ tensors, i.e. \textit{$r$-th secant varieties} of Segre varieties. Due to the independence of tensor rank by scalar multiplication, it is natural to consider these spaces projectively. 
\begin{definition}
    Fix integers $\bfn = (n_1,\ldots,n_k)$. The \textbf{Segre embedding} is the map of $\bbP\bbC^{n_1}\times\cdots\times\bbP\bbC^{n_k}$ defined by the line bundle $\calO_{\bbP\bbC^{n_1}\times\cdots\times\bbP\bbC^{n_k}}(1^k)$, i.e., 
    \[
    \begin{array}{r c c c}
        \nu_{\bfn} : & \bbP\bbC^{n_1}\times\cdots\times\bbP\bbC^{n_k} & \longrightarrow & \bbP\left(\bbC^{n_1}\otimes\cdots\otimes\bbC^{n_k}\right) \\
        & ([\bfv^1],\ldots,[\bfv^k]) & \mapsto & [\bfv^1\otimes\cdots\otimes\bfv^k].
    \end{array}
    \]
    The \textbf{Segre variety} is $X_{\bfn} = \im(\nu_\bfn)$. Its \textbf{$r$-th secant variety} is 
    \[
        \sigma_r(X_\bfn) = \overline{\left\{[T] ~:~ \rk(T) \leq r\right\}} \subset \bbP\left(\bbC^{n_1}\otimes\cdots\otimes\bbC^{n_k}\right). 
    \]
\end{definition}
We use the notation $n^k = (n,n,\ldots,n) \in \bbN^k$.

The first $r$ such that $[T] \in \sigma_r(X_{\bfn})$ is called the \textbf{border rank} of $[T]$:
\[
    \brk(T) = \min\{r ~:~ [T] \in \sigma_r(X_\bfn)\}.
\]

\begin{remark}
 
With an abuse of notation, we will denote by $\calD_T$ the decomposition locus of the tensor $T$ defined in \Cref{def:decomposition} also when regarded as subspace of the Segre variety product of projective spaces $\bbP\bbC^{n_1}\times\cdots\times\bbP\bbC^{n_k}$ or when regarded as a subset of the Segre variety $X_\bfn$; analogously for the forbidden locus. 
\end{remark}

By {\it conciseness} (see \cite[Proposition 3.1.3.1]{L}) or {\it autarky} (see \cite{BB17}),  if $T \in V_1 \otimes \cdots \otimes V_r$ where $V_i \subseteq \bbC^{n_i}$, then $\calD_T \subset V_1 \times \cdots \times V_r$. Hence, we will always consider tensors in their \emph{concise space}, namely $T \in V_1 \otimes\cdots\otimes V_k$ such that there is no $W_1 \otimes\cdots\otimes W_k \subsetneq V_1 \otimes\cdots\otimes V_k$ containing the tensor $T$. 
\begin{definition}
    For $\bfm=(m_1, \ldots , m_k)$, let ${\rm Sub}_{\bfm}$ be the space of tensors that, up to change of coordinates, lie in some $\PP(\bbC^{m_1}\otimes\cdots\otimes\bbC^{m_k})$.
\end{definition}
A way to understand if a tensor is concise is by looking at its \textit{flattenings}.
\begin{definition}
The $i$-th \textbf{flattening} of a tensor $T\in \CC^{n_1}\otimes \cdots \otimes \CC^{n_k}$ is the linear map 
\[
    \flatt_i(T) : (\bbC^{n_i})^\vee \rightarrow \bigotimes_{j\neq i} \CC^{n_j}
\]
given by contraction on the $i$-th factor. Sometimes, we will also refer to $\flatt_i(T)$ as the associated matrix in the standard basis.
\end{definition}

A classical construction is the one of \textit{dual varieties}. If we consider the Segre variety in the dual projective space $\calX_\bfn \subset \bbP (\bbC^{n_1}\otimes\cdots\otimes\bbC^{n_k})^\vee$, then, by biduality, we consider its dual variety $\calX_\bfn^\vee \subset  \bbP(\bbC^{n_1}\otimes\cdots\otimes\bbC^{n_k})$. This is a hypersurface if and only if $n_i -1\leq \sum_{j\neq i} n_j-1$, for all $i$.
In the following, we will mostly consider dual varieties of three-factors Segre varieties that are hypersurfaces. In such cases, the defining equation of $\calX_{n_1,n_2,n_3}^\vee$ is the so-called \emph{hyperdeterminant}, denoted by $H_{n_1,n_2,n_3}$. See \cite{GKZ,Ottaviani2013}.

\subsection{Spaces of tensors with finitely many orbits}\label{ssec:finite_orbits}
The group $\GL_{n_1}\times \cdots \times \GL_{n_k}$ naturally acts on $\bbC^{n_1}\times\cdots\times\bbC^{n_k}$. In \cite{K}, the author proved that the only spaces of tensors $\mathbb{C}^{n_1}\otimes \cdots \otimes \mathbb{C}^{n_k}$ with a finite number of $\big(\GL_{n_1}\times \cdots \times \GL_{n_k}\big)$-orbits are
\begin{enumerate}[label=\alph*)]
    \item $\mathbb{C}^m\otimes \mathbb{C}^n$;
    \item $\mathbb{C}^2\otimes \mathbb{C}^2\otimes \mathbb{C}^n$;
    \item $\mathbb{C}^2\otimes \mathbb{C}^3\otimes \mathbb{C}^n$.
\end{enumerate}
The case of matrices is well-understood and orbits correspond to matrices with equal rank.

In \cite{parfenov2001orbits}, normal forms for tensors in $\bbC^2 \otimes\bbC^b \otimes \bbC^c$ with $b \leq 3$ were provided. Recall that any three-factor (concise) tensor $T\in \CC^{2}\otimes \CC^{n_2}\otimes \CC^{n_3}$ can be studied as a pencil of $n_2\times n_3$ matrices by looking at the $2$-dimensional image of the first flattening. Explicitly, if we consider the standard basis $\{\bfe_1^i,\ldots,\bfe_{n_i}^i\}$ for the $i$-th factor of $\bbC^{n_1}\otimes\cdots\otimes\bbC^{n_k}$ and we write in coordinates $T=\sum_{i,j,k}t_{i,j,k}\bfe_i^1\otimes \bfe_j^2\otimes \bfe_k^3$, then we associate to $T$ the pencil of matrices $uA+vB$, where $A=(t_{1,i,j})$ and $B=(t_{2,i,j})$ are the matrices of size $n_2\times n_3$ obtained as images of $\bfe^1_1$ and $\bfe^1_2$ through $\flatt_1(T)$, respectively. Clearly, the notation as pencil of matrices of $T$ is not uniquely determined, but a pair of linearly independent matrices $A,B$ determines uniquely the pencil. In \cite{BL}, all ranks and border ranks of tensors in $\mathbb{C}^2\otimes \mathbb{C}^2\otimes \mathbb{C}^n$ and $\mathbb{C}^2\otimes \mathbb{C}^3\otimes \mathbb{C}^n$ were computed through the interpretation the so-called {\it Kronecker normal forms} of pencils of matrices and the theorem by Grigoriev \cite{grigoriev1978multiplicative}, Ja'Ja' \cite{jaja1979optimal} and Teichert \cite{teichert1986komplexitat} on ranks of pencils of matrices. 

We recall in \Cref{table:kronecker} the normal forms of tensors in $\bbC^2\otimes\bbC^{n_2}\otimes\bbC^{n_3}$ for $n_2 \leq 3$, following the numeration of \cite{BL}. 
\begin{notation*}
  
We denote by $T_n$ the $n$-th normal form of  \Cref{table:kronecker} and by $\calO_n$ the corresponding orbit. 
\end{notation*}
\begin{remark}\label{rmk:change_coordinates}
    Note that, since the tensor rank is invariant under the action of $\GL_{n_1}\times \cdots \times \GL_{n_k}$, if $T$ and $T'$ belong to the same orbit, then we can obtain the decomposition locus of one from the other. Say $T = (A_1, \ldots, A_k) \cdot T'$, for $A_i\in \GL_{n_i}$ then
    \begin{align*}
        \calD_T & = \{\bfa^1\otimes\cdots\otimes \bfa^k ~:~ (A_1^{-1}, \ldots, A_k^{-1}) \cdot (\bfa^1\otimes \cdots \otimes \bfa^k) \in \calD_{T'}\} \\ & = \{(A_1, \ldots, A_k)\cdot (\bfv_1 \otimes \cdots \otimes \bfv_k) ~:~ \bfv_1 \otimes \cdots \otimes \bfv_k \in \calD_{T'}\},
    \end{align*}
    and similarly for the forbidden loci. Therefore, besides the case of matrices and tangential tensors, in \Cref{sec:22n,sec:23n} we symbolically compute equations of decomposition loci only of tensors in normal forms. 
\end{remark}

{\footnotesize{
\begin{longtable}{c c c c c c}\label{table:kronecker}
    n. & Kronecker normal form & pencil & $\brk$ & $\rk$ & Ref. \\ \hline 
    
     \iddef{1} & $\bfe^1_1\otimes \bfe^2_1 \otimes \bfe^3_1$ & \footnotesize{$\left(\begin{smallmatrix} u \end{smallmatrix}\right)$} & 1 & 1 & \Cref{prop:matrices} \\
    
      \iddef{2} & $\bfe^1_1\otimes \bfe^2_1 \otimes \bfe^3_1 + \bfe_2^1\otimes \bfe_2^2 \otimes \bfe_1^3$ & \footnotesize{$\left(\begin{smallmatrix} u & v\end{smallmatrix}\right)$} & 2 & 2 & \Cref{prop:matrices}\\
    
       \iddef{3} & $\bfe^1_1\otimes \bfe^2_1 \otimes \bfe^3_1 + \bfe_2^1\otimes \bfe_2^2 \otimes \bfe_2^3$ & \footnotesize{$\left(\begin{smallmatrix} u & \\ & u \end{smallmatrix}\right)$} & 2 & 2 & \Cref{prop:matrices}\\
    
       \iddef{4} & $\bfe^1_1\otimes \bfe^2_1 \otimes \bfe^3_1 + \bfe_2^1\otimes \bfe_1^2 \otimes \bfe_2^3$ & \footnotesize{$\left(\begin{smallmatrix} u  \\ v \end{smallmatrix}\right)$} & 2 & 2 & \Cref{prop:matrices}\\
    
      \iddef{5} & $\bfe^1_1\otimes(\bfe^2_1 \otimes \bfe^3_1 + \bfe_2^2\otimes\bfe_2^3) + \bfe_2^1\otimes\bfe_1^2\otimes\bfe_2^3$ & \footnotesize{$\left(\begin{smallmatrix} u & v \\ & u\end{smallmatrix}\right)$} & 2 & 3 & \Cref{thm:tangential}\\
    
      \iddef{6} & $\bfe^1_1\otimes \bfe^2_1 \otimes \bfe^3_1 + \bfe^1_2\otimes \bfe^2_2 \otimes \bfe^3_2 $ & \footnotesize{$\left(\begin{smallmatrix} u & \\ & v\end{smallmatrix}\right)$} & 2 & 2 & \Cref{ssec:orbit_6}\\
    
      \iddef{7} & $\bfe^1_1\otimes(\bfe^2_1 \otimes \bfe^3_1 + \bfe_2^2\otimes\bfe_3^3) + \bfe_2^1\otimes\bfe_1^2\otimes\bfe_2^3$ & \footnotesize{$\left(\begin{smallmatrix} u & v & \\ & & u\end{smallmatrix}\right)$} & 3 & 3 & \Cref{ssec:orbit_7_8_11_12} \\
    
        \iddef{8} & $\bfe^1_1\otimes(\bfe^2_1 \otimes \bfe^3_1 + \bfe_2^2\otimes\bfe_2^3) + \bfe_2^1\otimes(\bfe_1^2\otimes\bfe_2^3 + \bfe_2^2 \otimes \bfe_3^3)$ & \footnotesize{$\left(\begin{smallmatrix} u & v & \\ & u & v\end{smallmatrix}\right)$} & 3 & 3 & \Cref{ssec:orbit_7_8_11_12}\\
    
       \iddef{9} & $\bfe^1_1\otimes(\bfe^2_1 \otimes \bfe^3_1 + \bfe_2^2\otimes\bfe_3^3) + \bfe_2^1\otimes(\bfe_1^2\otimes\bfe_2^3 + \bfe_2^2 \otimes \bfe_4^3)$ & \footnotesize{$\left(\begin{smallmatrix} u & v & & \\ & & u & v\end{smallmatrix}\right)$} & 4 & 4 & \Cref{ssec:orbit_9}\\
    
       \iddef{10} & $\bfe^1_1\otimes(\bfe^2_1 \otimes \bfe^3_1 + \bfe_2^2\otimes\bfe_2^3 + \bfe_3^2\otimes\bfe_3^3)$ & \footnotesize{$\left(\begin{smallmatrix} u & &  \\ & u &   \\ & & u \end{smallmatrix}\right)$} & 3 & 3 & \Cref{prop:matrices}\\
    
        \iddef{11}  & $\bfe^1_1\otimes(\bfe^2_1 \otimes \bfe^3_1 + \bfe_3^2\otimes\bfe_2^3) + \bfe_2^1 \otimes \bfe_2^2\otimes\bfe_1^3$ & \footnotesize{$\left(\begin{smallmatrix} u &  \\ v &   \\ & u \end{smallmatrix}\right)$} & 3 & 3 & \Cref{ssec:orbit_7_8_11_12}\\

     \iddef{12} & $\bfe^1_1\otimes(\bfe^2_1 \otimes \bfe^3_1 + \bfe_2^2\otimes\bfe_2^3) + \bfe_2^1 \otimes (\bfe_2^2\otimes\bfe_1^3 + \bfe_3^2 \otimes \bfe_2^3)$ & \footnotesize{$\left(\begin{smallmatrix} u &  \\ v & u \\ & v \end{smallmatrix}\right)$} & 3 & 3 & \Cref{ssec:orbit_7_8_11_12}\\

       \iddef{13}  &  $\bfe^1_1\otimes(\bfe^2_1 \otimes \bfe^3_1 + \bfe_2^2\otimes\bfe_3^3) + \bfe_2^1 \otimes (\bfe_1^2\otimes\bfe_2^3 + \bfe_3^2 \otimes \bfe_3^3)$ & \footnotesize{$\left(\begin{smallmatrix} u & v &  \\ & & u \\ & & v \end{smallmatrix}\right)$} & 3 & 4 & \Cref{ssec:n13}\\

    \iddef{14} & $\bfe^1_1\otimes(\bfe^2_1 \otimes \bfe^3_1 + \bfe_2^2\otimes\bfe_2^3) + \bfe_2^1 \otimes  \bfe_3^2 \otimes \bfe_3^3$ & \footnotesize{$\left(\begin{smallmatrix} u & &  \\  & u & \\ & & v \end{smallmatrix}\right)$} & 3 & 3 & \Cref{ssec:23n_rank3_border3}\\

     \iddef{15} & $\bfe^1_1\otimes(\bfe^2_1 \otimes \bfe^3_1 + \bfe_2^2\otimes\bfe_2^3 + \bfe_3^2 \otimes \bfe_3^3) + \bfe_2^1 \otimes  \bfe_1^2 \otimes \bfe_2^3$ & \footnotesize{$\left(\begin{smallmatrix} u & v &  \\  & u & \\ & & u \end{smallmatrix}\right)$} & 3 & 4 & \Cref{ssec:n15}\\

       \iddef{16} & $\bfe^1_1\otimes(\bfe^2_1 \otimes \bfe^3_1 + \bfe_2^2\otimes\bfe_2^3 + \bfe_3^2 \otimes \bfe_3^3) + \bfe_2^1 \otimes (\bfe_1^2 \otimes \bfe_2^3 + \bfe_2^2 \otimes \bfe_3^3)$ & \footnotesize{$\left(\begin{smallmatrix} u & v &  \\  & u & v \\ & & u \end{smallmatrix}\right)$} & 3 & 4 & \Cref{ssec:n16}\\

       \iddef{17} & $\bfe^1_1\otimes(\bfe^2_1 \otimes \bfe^3_1 + \bfe_2^2\otimes\bfe_2^3) + \bfe_2^1 \otimes (\bfe_1^2 \otimes \bfe_2^3 + \bfe_3^2 \otimes \bfe_3^3)$ & \footnotesize{$\left(\begin{smallmatrix} u & v &  \\  & u & \\ & & v \end{smallmatrix}\right)$} & 3 & 4 & \Cref{ssec:n17}\\

     \iddef{18} & $\bfe^1_1\otimes(\bfe^2_1 \otimes \bfe^3_1 + \bfe_2^2\otimes\bfe_2^3) + \bfe_2^1 \otimes (\bfe_2^2 \otimes \bfe_2^3 + \bfe_3^2 \otimes \bfe_3^3)$ & \footnotesize{$\left(\begin{smallmatrix} u & &  \\  & u+v & \\ & & v \end{smallmatrix}\right)$} & 3 & 3 & \Cref{ssec:23n_rank3_border3} \\

       \iddef{19} & $\bfe^1_1\otimes(\bfe^2_1 \otimes \bfe^3_1 + \bfe_2^2\otimes\bfe_2^3 + \bfe_3^2 \otimes \bfe_4^3) + \bfe_2^1 \otimes (\bfe_1^2 \otimes \bfe_2^3 + \bfe_2^2 \otimes \bfe_3^3)$ & \footnotesize{$\left(\begin{smallmatrix} u & v &  \\  & u & v \\ & & & u \end{smallmatrix}\right)$} & 4 & 4 & \Cref{ssec:n19}\\

       \iddef{20} & $\bfe^1_1\otimes(\bfe^2_1 \otimes \bfe^3_1 + \bfe_2^2\otimes\bfe_3^3 + \bfe_3^2 \otimes \bfe_4^3) + \bfe_2^1 \otimes \bfe_1^2 \otimes \bfe_2^3$ & \footnotesize{$\left(\begin{smallmatrix} u & v &  \\  & & u & \\ & & & u \end{smallmatrix}\right)$} & 4 & 4 & \Cref{ssec:n20}\\

       \iddef{21} & $\bfe^1_1\otimes(\bfe^2_1 \otimes \bfe^3_1 + \bfe_2^2\otimes\bfe_3^3 + \bfe_3^2 \otimes \bfe_4^3) + \bfe_2^1 \otimes (\bfe_1^2 \otimes \bfe_2^3 + \bfe_2^2 \otimes \bfe_4^3)$ & \footnotesize{$\left(\begin{smallmatrix} u & v &  \\  & & u & v \\ & & & u \end{smallmatrix}\right)$} & 4 & 5 & \Cref{ssec:n21}\\

       \iddef{22} & $\bfe^1_1\otimes(\bfe^2_1 \otimes \bfe^3_1 + \bfe_2^2\otimes\bfe_3^3) + \bfe_2^1 \otimes (\bfe_1^2 \otimes \bfe_2^3 + \bfe_3^2 \otimes \bfe_4^3)$ & \footnotesize{$\left(\begin{smallmatrix} u & v &  \\  & & u & \\ & & & v \end{smallmatrix}\right)$} & 4 & 4 & \Cref{ssec:n22}\\

      \iddef{23} & $\bfe^1_1\otimes(\bfe^2_1 \otimes \bfe^3_1 + \bfe_2^2\otimes\bfe_2^3 + \bfe_3^2 \otimes \bfe_3^3) + \bfe_2^1 \otimes (\bfe_1^2 \otimes \bfe_2^3 + \bfe_2^2 \otimes \bfe_3^3 + \bfe_3^2 \otimes\bfe_4^3)$ & \footnotesize{$\left(\begin{smallmatrix} u & v &  \\  & u & v & \\ & & u & v \end{smallmatrix}\right)$} & 4 & 4 & \Cref{ssec:n23}\\

     \iddef{24} & $\bfe^1_1\otimes(\bfe^2_1 \otimes \bfe^3_1 + \bfe_2^2\otimes\bfe_3^3 + \bfe_3^2 \otimes \bfe_5^3) + \bfe_2^1 \otimes (\bfe_1^2 \otimes \bfe_2^3 + \bfe_2^2 \otimes \bfe_4^3)$ & \footnotesize{$\left(\begin{smallmatrix} u & v &  \\  & & u & v & \\ & & & & u \end{smallmatrix}\right)$} & 5 & 5 & \Cref{ssec:n24_n25}\\

       \iddef{25} & $\bfe^1_1\otimes(\bfe^2_1 \otimes \bfe^3_1 + \bfe_2^2\otimes\bfe_2^3 + \bfe_3^2 \otimes \bfe_4^3) + \bfe_2^1 \otimes (\bfe_1^2 \otimes \bfe_2^3 + \bfe_2^2 \otimes \bfe_3^3 + \bfe_3^2 \otimes\bfe_5^3)$ & \footnotesize{$\left(\begin{smallmatrix} u & v &  \\  & u & v & \\ & & & u & v \end{smallmatrix}\right)$} & 5 & 5 & \Cref{ssec:n24_n25}\\

     \iddef{26} & $\bfe^1_1\otimes(\bfe^2_1 \otimes \bfe^3_1 + \bfe_2^2\otimes\bfe_3^3 + \bfe_3^2 \otimes \bfe_5^3) + \bfe_2^1 \otimes (\bfe_1^2 \otimes \bfe_2^3 + \bfe_2^2 \otimes \bfe_4^3 + \bfe_3^2 \otimes\bfe_6^3)$ & \footnotesize{$\left(\begin{smallmatrix} u & v &  \\  & & u & v & \\ & & & & u & v \end{smallmatrix}\right)$} & 6 & 6 & \Cref{ssec:n26}\\
    ~ \\ 
    \caption{The 26 normal forms in $\bbC^2 \otimes \bbC^b \otimes \bbC^c$ with $b\leq3$ or $c\leq 6$.}
\end{longtable}
}}

\section{Matrices}\label{sec:matrices}
As a first example of computation, we describe the case of matrices, generalizing \cite[Corollary 3.2]{CCO} which considered only the symmetric case. This section covers the orbits n.\idref{1}, \idref{2}, \idref{3}, \idref{4} and n.\idref{10} of \Cref{table:kronecker}.

Let $A \in  \Mat_{m\times n}(\bbC)$ be a $m\times n$ matrix of rank $r$ and consider its thin Singular Value Decomposition (SVD) $A = U \Sigma V^*$, where $U\in \Mat_{m\times r}(\mathbb{C})$ and $V\in \Mat_{n\times r}(\mathbb{C})$ are unitary matrices, $V^*$ denotes the transpose conjugate and $\Sigma = {\rm diag}(\sigma_1,\ldots,\sigma_r) \in \Mat_{r\times r}(\mathbb{C})$ is a diagonal matrix whose diagonal contains the ordered singular values $\sigma_1 \geq \cdots \geq \sigma_r$ of $A$. Then, 
\[
	A = \sum_{i=1}^r \sigma_i \bfu_i \otimes \bfv_i
\]
where $\bfu_i\in \mathbb{C}^m$ and $\bfv_i\in \mathbb{C}^n$ indicate the $i$th column and row of the matrices $U$ and $V^*$, respectively. In particular, we have that $A$ is concise in the subspace
\begin{equation*}
    \bbU \otimes \bbV :=\langle \bfu_1,\ldots,\bfu_r \rangle \otimes \langle \bfv_1,\ldots,\bfv_r \rangle  \subseteq \bbC^{m} \otimes \bbC^n.
\end{equation*} 
The matrix $A^+ := V\Sigma^{-1} U^*$ is the \textit{pseudoinverse} of $A$, where $\Sigma^{-1} := {\rm diag}(\frac{1}{\sigma_1},\ldots,\frac{1}{\sigma_r}) \in \Mat_{r\times r}(\mathbb{C})$.

\begin{proposition}\label{prop:matrices} Let $\bbU \otimes \bbV$ be the concision space of $A \in \Mat_{m\times n}(\bbC)$. The decomposition locus of  $A$ is
\[
	\mathcal{D}_A = \left\{ (\bf{u}, \bf{v}) \in \bbU\times \bbV \subseteq \bbC^m \times \bbC^n ~:~ \bf{v}^T A^+ \bf{u} \neq 0\right\}.
\]
\end{proposition}
\begin{proof}
As mentioned in the introduction, by either concision for tensors (cf. \cite[Proposition 3.1.3.1]{L}) or autarky property (cf. \cite{BB17}), $\calD_A \subseteq \bbU \times \bbV$. We first compute the decomposition locus of the rank-$r$ matrix $\Sigma \in \Mat_{r\times r}(\bbC)$ from the thin SVD of $A = U \Sigma V^*$. By linear algebra computation,
\[
	\det(\Sigma + \lambda \bfx \otimes \bfy) = \det(\Sigma)(1+\lambda \bfy^T \Sigma^{-1} \bfx) = \sigma_1\cdots \sigma_r + \lambda \sum_{i=1}^r \frac{1}{\sigma_i}x_iy_i.
\]
Hence, $\calD_\Sigma = \{(\bfx,\bfy) \in \bbC^r \times \bbC^r ~:~ \bfy^T \Sigma^{-1} \bfx \neq 0\}.$ Now, since $A^+ = V\Sigma^{-1}U^*$, we consider the corresponding change of variables to obtain $\calD_A$ from $\calD_\Sigma$ as explained in \Cref{rmk:change_coordinates}, i.e.,
\begin{align*}
	\calD_A &= \{(\bfu,\bfv) \in \bbU \times \bbV ~:~ (U^*\bfu,V^T\bfv ) \in \calD_\Sigma \} \\
	& =   \{(\bfu,\bfv) \in \bbU \times \bbV ~:~ \bfv^T V \Sigma^{-1} U^*\bfu \neq 0 \} = \{(\bfu,\bfv) \in \bbU \times \bbV ~:~ \bfv^T A^+ \bfu \neq 0 \}. \qedhere
\end{align*}
\end{proof}

\section{Tangential tensors}\label{sec:tangential}
We consider here \textit{tangential tensors}, namely tensors lying on the tangential variety of the Segre variety but not on the Segre variety itself. In the literature of quantum information, these are known as (tensors in the orbits of) {\it $W$-states}, see e.g. \cite{ballico2019partially} and the references therein. This case corresponds to the orbit n.\idref{5} from \Cref{table:kronecker}. 
\begin{definition}
    The \textbf{tangential variety} of the Segre variety is 
    \[
        \tau(X_\bfn) = \overline{\left\{[T] ~:~ \exists \, q \in X_\bfn \text{ such that } [T] \in T_qX_\bfn\right\}},
    \]
    where $T_qX_\bfn$ denotes the tangent space to the Segre variety at $q$.
\end{definition}
If $Q = \bfv^1 \otimes \cdots \otimes \bfv^k \in X_\bfn$ then it is immediate to see that the all tensors in the tangent space at $q = [Q] \in X_\bfn$ can be written as $T = \sum_{i=1}^k \bfv^1 \otimes \cdots \otimes \bfw^i \otimes \cdots \otimes \bfv^k$, for some $\bfw^i \in \bbC^{n_i}$. In particular, $T \in \Sub_{2^k}$.

\begin{remark}\label{remark:tg}
    Let $q \in X_{2^k}$ and $[T] \in T_q(X_{2^k}) \smallsetminus \{q\}$. Then, there exists a $0$-dimensional scheme of length two $J \subset (\bbP^1)^{\times k}$ supported at $q$ such that $[T] \in \langle \nu_{2^k}(J) \rangle$. Such a scheme is classically called a {\it $2$-jet} (cf. \cite{AH1997jets, Bernardi2009982}). By \cite[Theorem 1]{BB13}, the rank of $T$ is determined by the conciseness of $J$; in particular, since we are always assuming that $T$ is concise, $\rk(T) = k$. 
\end{remark}

We prove that the only forbidden point for tangential tensors is the tangency point. This generalizes the analogous result for symmetric tensors in \cite{CCO}.

\begin{theorem}\label{thm:tangential}
    Fix $q \in X_\bfn$, $\bfn = (n_1,\ldots,n_k)$ with $k \geq 3$. Let $[T] \in T_q X_\bfn \smallsetminus \{q\}$. Then, $\calF_T = \{q\}$.
\end{theorem}

The latter theorem gives a very immediate way to check if a rank-one tensor is forbidden for a tangential tensor. Indeed, if $T$ is tangential tensor at $q = [Q]$, then the image of all flattenings $\flatt_i(T)$ contains only one rank-one tensor: that is exactly the image by $\flatt_i(T)$ of the annihilator by contraction of the $i$-th factor of $Q$ in $(\bbC^{n_i})^\vee$.  See \Cref{algo:tangential}.

\begin{algorithm}[H]
    \caption{Test if a rank-$1$ tensor is in the forbidden locus of a tangential tensor}
    \flushleft{
    \textbf{Input:} }
    \begin{itemize}
        \item Tangential tensor $T \in \bbC^{n_1}\otimes\cdots\otimes\bbC^{n_k}$ with $k \geq 2$.
        \item Rank-one tensor $P = \bfp^1\otimes\cdots\otimes\bfp^k \in \bbC^{n_1}\otimes\cdots\otimes\bbC^{n_k}$.
    \end{itemize}
    \flushleft{
    \textbf{Output:} Whether $P \in \calF_T$ or $P \not\in \calF_T$.
    }
    
    \flushleft{
    \textbf{Procedure:}
    }
    \begin{itemize}
        \item For any $i = 1,\ldots,k$, let $(\bfp^i)' \in (\bbC^{n_i})^\vee$ be an annihilator of $\bfp^i$ by contraction in $\bbC^{n_i}$.
        \item If, for any $i = 1,\ldots,k$, $\rk \left(\flatt_i(T)((\bfp^i)')\right) > 1$, then $P \not\in \calF_T$; otherwise, $P \in \calF_T$.
    \end{itemize}
    \label{algo:tangential}
\end{algorithm}

First of all we show that the tangency point is indeed forbidden.

\begin{lemma}\label{lemma:tangency_is_forbidden}
    Fix $q \in X_\bfn$, $\bfn = (n_1,\ldots,n_k)$ with $k \geq 3$. Let $[T] \in T_q X_\bfn \smallsetminus \{q\}$. Then, $q \in \calF_T$.
\end{lemma}
\begin{proof}
    Without loss of generalities, we may assume $\bfn = 2^k$ and $T$ concise in $(\bbC^2)^{\otimes k}$. In particular, the $2$-jet $J$ supported at $q$ such that $[T] \in \langle \nu_{2^k}(J)\rangle$ as in \Cref{remark:tg} is also concise. Therefore, if $q = [Q]$, since $\langle \nu_{2^k}(J)\rangle$ is exactly the line $\langle [T],q\rangle$, then the tensor $T-\lambda Q$ is also tangential and concise for any $\lambda \in \bbC$. In particular, $\rk(T-\lambda Q) = \rk(T)$ for any $\lambda$ and $q \in \calF_T$.
\end{proof}
To finish the proof of \Cref{thm:tangential} it remains to show that $Q$ is the only point in $\calF_T $. We distinguish the case $k=3$ from all the others.

\begin{proposition}\label{prop:tangent_k=3}
    \Cref{thm:tangential} holds for $k = 3$. 
\end{proposition}
\begin{proof}
    Recall that the tangential variety of $X_{222} \subset \bbP^7$ is defined by the vanishing of the Cayley equation (cf. \cite{GKZ}) and it is the closure of the set of rank-$3$ tensors in $\mathbb{P}(\mathbb{C}^2)^{\otimes 3}$. Let $p = [P] = [\bfa \otimes \bfb \otimes \bfc] \in X_{222}$, with $\bfa = a_1\bfe^1_{1} + a_2\bfe^1_{2}, \bfb = b_1\bfe^2_{1} + b_2\bfe^2_{2}, \bfc = c_1\bfe^3_{1} + c_2\bfe^3_{2}$. 
    If $p \in \calF_T$ then $T-\lambda P$ has rank three for all $\lambda \in \bbC$, in particular, the Cayley's hyperdeterminant of $T-\lambda P$ should vanish for all $\lambda \in \bbC$. Up to change of coordinates, we assume $Q = \bfe^1_{1} \otimes \bfe^2_{1} \otimes \bfe^3_{1}$ and 
    \[
        T = \bfe^1_{2} \otimes \bfe^2_{1}\otimes \bfe^3_{1} + \bfe^1_{1}\otimes \bfe^2_{2}\otimes \bfe^3_{1} + \bfe^1_{1}\otimes \bfe^2_{1}\otimes \bfe^3_{2} \in T_qX_{222}.
    \]
    By evaluating the Cayley's hyperdeterminant $H_{222}$ (cf. \cite[Ch. 14, Proposition 1.7]{GKZ}) at $T-\lambda P$, we get
    \begin{align}\label{eq:CayleyEval}
        H_{222}(T-\lambda P) = \lambda^2&\left(a_2^2b_2^2c_1^2+ 2a_2^2b_1b_2c_1c_2+2a_1a_2b_2^2c_1c_2+a_2^2b_1^2c_2^2+2a_1a_2b_1b_2c_2^2+a_1^2b_2^2c_2^2\right)\\ & -4\lambda a_2b_2c_2 \nonumber.
    \end{align}
    If $P \in \calF_T$, then \cref{eq:CayleyEval} is identically zero as univariate polynomial in $\bbC[\bfa,\bfb,\bfc][\lambda]$. Therefore, a necessary condition is $a_2b_2c_2 = 0$. If $a_2 = 0$, then $H_{222}(T-\lambda P) = \lambda^2 a_1^2b_2^2c_2^2$: since $\bfa \neq 0$, then either $b_2 = 0$ or $c_2 = 0$. Therefore, since \cref{eq:CayleyEval} is invariant under permutation of $\bfa, \bfb, \bfc$, we deduce that 
    \begin{equation}\label{eq:tangent_inclusion}
        \calF_T \subset \left( \bfe^1_{1} \times \bfe^2_{1} \times \bbP^1 \right) \cup \left( \bfe^1_{1} \times \bbP^1 \times \bfe^3_{1} \right) \cup \left( \bbP^1 \times \bfe^2_{1} \times \bfe^3_{1} \right).
    \end{equation}
    Now, let $p = [P]$ be an element of the right-hand side of \cref{eq:tangent_inclusion}. Say $P = \bfe_1^1 \otimes \bfe_1^2 \otimes (\alpha \bfe_1^3 + \beta \bfe_2^3)$ with $(\alpha:\beta) \in \bbP^1$. If $\beta \neq 0$, then 
    \[
        T-\frac{1}{\beta}P = \left( \left(-\frac{\alpha}{\beta}\bfe^1_1+\bfe^1_2\right)\otimes \bfe^2_1 + \bfe^1_{1} \otimes \bfe^2_{2}\right) \otimes \bfe^3_{1},
    \]
    i.e., $P \in \calD_T$. Hence, $\calF_T \subset \{q\}$. The fact that $q \in \calF_T$ is \Cref{lemma:tangency_is_forbidden}. 
\end{proof}

\begin{proposition}\label{prop:tangent_alldifferent}
    Fix $q = [\bfv^1\otimes\cdots\otimes \bfv^k]\in X_\bfn$, $\bfn = (n_1,\ldots,n_k)$ with $k \geq 3$. Let $[T] \in T_q X_\bfn \smallsetminus \{q\}$. Consider $p = [\bfp^1\otimes\cdots\otimes\bfp^k]$ such that $[\bfp^i] \neq [\bfv^i]$ for all $i = 1,\ldots,k$. Then, $p \in \calD_T$.
\end{proposition}
\begin{proof}
    Up to change of coordinates, we may assume $q = [\bfe_1^1\otimes\cdots\otimes \bfe_1^k]$ and $T = \sum_{i=1}^k \bfe_1^1 \otimes\cdots\otimes \bfe_2^i \otimes\cdots\otimes \bfe_1^k$.

We construct a curve $\calC_k$ of multidegree $1^k$, i.e. a rational normal curve of degree $k$  passing through $p$ and such that $T$ belongs to the tangent line to $\calC_k$ at $q$. This will show that $p$, belonging to the rational normal curve to which $T$ is tangent and being distinct from the point of tangency itself (the only prohibited point for tangent points to a rational normal curve), is a point of decomposition for $T$.

We may assume that $\bfp^i = p_i\bfe_1^i + \bfe_2^i$ for $i = 1,\ldots,k$. Consider the projective linear map $f_i : \bbP\bbC^2 \to \bbP\bbC^2$ defined by the matrix $\left(\begin{smallmatrix}
        1 & p_i \\ 0 & 1
    \end{smallmatrix}\right)$, i.e.,
    \begin{equation}\label{eq:RNC}
        f_i : \bbP^1 \to \bbP^1, (x:y) \mapsto (x+p_iy:y).
    \end{equation}
    Let $\calC_k$ be the image of $f = \nu_{2^k} \circ (f_1\times\cdots\times f_k) : \bbP^1 \to (\bbP^1)^{\times k} \to X_{2^k}$. The points $p,q \in \calC_k$ since $f(1:0) = q$ and $f(0:1) = p$. The tangent line to $\calC_k$ at $q$ contains the tensor $T$: consider the limit of the secant lines $\langle f(1:0),f(1:t) \rangle$ for $t \mapsto 0$ and compute the linear part of the Taylor expansion of $f(1:t)$:   \begin{align*}    
        {\rm coeff}_t(f(1:t)) & = {\rm coeff}_t\left(((1+p_1t)\bfe_1^1 + t\bfe_2^1)\otimes\cdots\otimes ((1+p_kt)\bfe_1^k + t\bfe_2^k)\right) \\
        & = \left(\sum_i p_i\right) \bfe_1^1\otimes\cdots\otimes\bfe_1^k + T \in \langle q,T \rangle.
    \end{align*}
    Therefore the rational normal curve $\calC_k$ satisfies all assumptions of  \cite[Theorem 3.3]{CCO}, so one can conclude that $p \in\calD_T$. 
\end{proof}

\begin{remark}
    Our proof of \Cref{prop:tangent_alldifferent} allows for a more general statement for {\it partially-symmetric tensors}, e.g. see \cite{bernardi2018hitchhiker} for definitions. If the tensor $T$ and the rank-one tensor $\bfp$ in the statement of \Cref{prop:tangent_alldifferent} share the same partial symmetries, i.e., they belong to the same space ${\rm Sym}^{d_1}(\bbC^{n_1})\otimes\cdots\otimes {\rm Sym}^{d_m}(\bbC^{n_m}) {\subset (\bbC^{n_1})^{\otimes d_1} \otimes\cdots\otimes (\bbC^{n_m})^{\otimes d_m}}$, then $\bfp$ appears in a minimal rank partially-symmetric decomposition of $T$. {Indeed, by definition of tangential tensors, there is $(A_1,\ldots,A_m) \in \GL_{n_1}\times\cdots\times\GL_{n_m}$ such that $T' = (A_1,\ldots,A_m) \cdot T$ is symmetric: note that, $\bfp' = (A_1,\ldots,A_m) \cdot \bfp \in {\rm Sym}^{d_1}(\bbC^{n_1})\otimes\cdots\otimes {\rm Sym}^{d_m}(\bbC^{n_m})$ is not necessary symmetric. Now, if we follow the proof of \Cref{prop:tangent_alldifferent} for $T'$ and $\bfp'$, we notice that the parametrization \cref{eq:RNC} respects the same partial symmetry, i.e., $\calC_k$ is contained in the same space of partially symmetric tensors. Therefore, we construct a partially symmetric minimal decomposition of $T'$ which involves $\bfp'$. By applying the inverse transformation $(A^{-1}_1,\ldots,A^{-1}_m)  \in \GL_{n_1}\times\cdots\times\GL_{n_m}$, we obtain a partially symmetric minimal decomposition of $T$ which involves $\bfp$.}

    We like to stress that the proof of \Cref{prop:tangent_alldifferent}, combined with classical apolarity theory for binary forms, is constructive. We give an explicit example for third-order tensors.
\end{remark}
    
\begin{example}
    Let $T = \bfe_1^1\otimes\bfe_1^2\otimes\bfe_2^3 + \bfe_1^1\otimes\bfe_2^2\otimes\bfe_1^3 + \bfe_2^1\otimes\bfe_1^2\otimes\bfe_1^3$ be a tangential tensor in its normal form. Consider $P = (\bfe_1^1 + \bfe^1_2) \otimes (\bfe^2_1 + \bfe^2_2)\otimes \bfe^3_2$. Then, the rational normal curve $\calC_3 \subset X_{222}$ defined by
    \[
        f : (a:b) \in \bbP^1 \mapsto [(a+b,b) \otimes (a+b,b) \otimes (a,b)] \in X_{222}
    \]
    i.e.,
    $
        f(a:b) = \left((a+b)^2a:(a+b)^2b:(a+b)ab:(a+b)b^2:(a+b)ab:(a+b)b^2:ab^2:b^3\right) \in \bbP^7,
    $
    passes through $p = [P]$ and contains $T$ on the tangent line at $q = [\bfe^1_1\otimes\bfe^2_1\otimes\bfe^3_1]$. The curve $\calC_3$ is the section of $X_{222}$ by the linear space $H = \{X_{101}-X_{110}-X_{111} = X_{011} - X_{110} - X_{111} = X_{010} - X_{100} = X_{001} - X_{100} - X_{110} - X_{111} = 0\}$. Therefore, on $H$, we can use the coordinates $Y_0 := X_{000}, Y_1 = X_{100}, Y_2 = X_{110}, Y_3 = X_{111}$. In order to make the parametrization of the curve $\calC_3$ in $H$ in the classical monomial form of Veronese embedding, we actually consider the change of coordinates given by 
    \[
        W_0 := Y_0 - 2Y_1 + Y_2, \quad W_1 := Y_1 - Y_2, \quad W_2 := Y_2, \quad W_3 := Y_3. 
    \]
    In these coordinates, we identify $H$ with the space of binary cubics in $\bbC[x,y]$ by associating to $W_i$ the $i$-th element of the monomial basis, i.e., $W_i = {3 \choose i}x^{3-i}y^i$. Namely, $\calC_3$ is the curve of powers of linear forms, the tensor $T$ has coordinates $(-2:1:0:0)$ and corresponds to the binary cubic $F = -2x^3+3x^2y$. The tensor $P$ instead has coordinates $(0:0:0:1)$ and corresponds to $y^3$. By classical apolarity theory, minimal decompositions of homogeneous polynomials as sums of powers of linear forms are constructed by looking at sets of points of minimal cardinality whose defining ideal is contained in the annihilator ideal ${\rm Ann}(F) = \{g \in \bbC[\alpha,\beta] ~:~ g(\partial_x,\partial_y) \circ F = 0\}$, see \cite[Lemma 1.15]{iarrobino1999power} or \cite[Lemma 5]{bernardi2018hitchhiker}. The coordinates of the points give the coefficients of the summands of minimal decomposition. In our case: ${\rm Ann}(F) = (\alpha^2, \alpha^3 + 2\alpha^2\beta)$. Since $F$ has rank three, we want to find a cubic polynomial in ${\rm Ann}(F)$ having distinct roots and divisible by $\alpha$ (i.e., one of the roots is $(0:1)$). Such a cubic is of the form $\alpha(\alpha^2 + 2\alpha\beta + \mu \beta^2)$ where $\mu$ is general, i.e., it is such that the quadratic factor has distinct roots. For example, if we consider $\mu = -3$, then $\alpha(\alpha-\beta)(\alpha+3\beta) \in {\rm Ann}(F)$. That means that a minimal rank decomposition of $T$ is given by $\{f(0:1), f(1:1), f(-3:1)\}$. It is easy linear algebra to see that indeed 
    \[
        T = -\frac{1}{3}P + \frac{1}{4}(2\bfe_1^1 + \bfe^1_2) \otimes (2\bfe^2_1 + \bfe^2_2)\otimes (\bfe^3_1 + \bfe^3_2) + \frac{1}{12}(-2\bfe_1^1 + \bfe^1_2) \otimes (-2\bfe^2_1 + \bfe^2_2)\otimes (-3\bfe^3_1 + \bfe^3_2).
    \]
    Similar construction can be pursued whenever we are in the assumption of \Cref{prop:tangent_alldifferent}.
\end{example}

We are now ready to complete the proof of the main theorem of this section. 

\begin{proof}[Proof of \Cref{thm:tangential}]
    Without loss of generalities, we may assume $\bfn = 2^k$ and $T$ concise in $(\bbC^2)^{\otimes k}$. In particular, the $2$-jet $J$ supported at $q$ such that $[T] \in \langle \nu_{2^k}(J)\rangle$ as in \Cref{remark:tg} is also concise. Up to change of coordinates, we may assume $q = [\bfe_1^1\otimes\cdots\otimes \bfe_1^k]$ and $T = \sum_{i=1}^k \bfe_1^1 \otimes\cdots\otimes \bfe_2^i \otimes\cdots\otimes \bfe_1^k$. 

    By \Cref{lemma:tangency_is_forbidden}, $q \in \calF_T$.
    
    Vice versa, let $p = [\bfp^1\otimes\cdots\otimes\bfp^k]$. If $[\bfp^i] \neq [\bfe_1^i]$ for all $i = 1,\ldots,k$ then we conclude by \Cref{prop:tangent_alldifferent}. Assume that $[\bfp^i] = [\bfe_1^i]$ for some $i$: without loss of generalities, up to reordering, we may assume that equality holds for $i = 1,\ldots,m$. Recall that $m < k$ because $[p] \neq [q]$. 

    If $k - m \geq 3$, then we write
    \begin{align*}
        T + \lambda \bfp^1\otimes\cdots\otimes \bfp^k & = \sum_{i=1}^m \bfe_1^1 \otimes \cdots \otimes \bfe_2^i \otimes \cdots \otimes \bfe_1^k + \left(\bfe_1^1 \otimes \cdots \otimes \bfe_1^m\right) \otimes T',
    \end{align*}
    where $T' = \sum_{i=m+1}^k \bfe_1^{m+1} \otimes \cdots \otimes \bfe_2^i \otimes \cdots \otimes \bfe_1^k + \lambda \bfp^{m+1} \otimes \cdots \otimes \bfp^k$. Since $T'$ and $\bfp^{m+1} \otimes \cdots \otimes \bfp^k$ satisfy the assumptions of \Cref{prop:tangent_alldifferent}, there exists some $\lambda \neq 0$ such that the rank of $T'$ is $k - m - 1$. Therefore, for such $\lambda$, we have that $T + \lambda \bfp^1\otimes\cdots\otimes \bfp^k$ has rank $k-1$, i.e., $p \not\in \calF_T$.

    If $k - m < 3$, then we consider 
    \begin{align*}
        T + \lambda \bfp^1\otimes\cdots\otimes \bfp^k & = \sum_{i=1}^{k-3} \bfe_1^1 \otimes \cdots \otimes \bfe_2^i \otimes \cdots \otimes \bfe_1^k + \left(\bfe_1^1 \otimes \cdots \otimes \bfe_1^{k-3}\right) \otimes T',
    \end{align*}
    where $T' = \sum_{i=k-2}^k \bfe_1^{k-2} \otimes \cdots \otimes \bfe_2^i \otimes \cdots \otimes \bfe_1^k + \lambda \bfp^{k-2} \otimes \cdots \otimes \bfp^k$. Now, we cannot apply \Cref{prop:tangent_alldifferent} because $m \geq k - 2$, i.e., $[\bfp^{k-2}] = [\bfe_1^{k-2}]$. However, $T'$ and $\bfp^{k-2} \otimes \cdots \otimes \bfp^k$ satisfy the assumptions of \Cref{prop:tangent_k=3}. Hence, there exists some $\lambda \neq 0$ such that the rank of $T'$ is $2$. Therefore, for such lambda we have that $T + \lambda \bfp^1\otimes\cdots\otimes \bfp^k$ has rank $k-1$, i.e., $p \not\in \calF_T$.
\end{proof}

\section{Tensors in $\bbC^2 \otimes \bbC^2 \otimes \bbC^n$}\label{sec:22n}

In this section, we explain how to compute the decomposition locus of tensors whose normal form is concise in $\bbC^2 \otimes \bbC^2 \otimes \bbC^n$, for $n \geq 2$. The case of tangential tensors (orbit n.\idref{5}) has been treated in \Cref{sec:tangential}. Hence, accordingly to the normal forms explained in \Cref{ssec:finite_orbits}, we only have the following cases. 

\subsection{Orbit n.\protect\idref{6}: concise tensors of rank two}\label{ssec:orbit_6}

If $T \in \bbC^2 \otimes \bbC^2 \otimes \bbC^n$, with $n \geq 2$, is concise of rank $2$, then $T \in \Sub_{222}$. By Kruskal criterion \cite{kruskal}, it is known that tensors in $\Sub_{222}$ are \textit{identifiable}, i.e., there is a unique way to write $T = P_1 + P_2$. See also \cite[Proposition 2.3]{BBS}. Hence, $\calD_T = \{[P_1],[P_2]\}$.

\subsection{Orbits n.\protect\idref{7}, n.\protect\idref{8}, n.\protect\idref{11}, n.\protect\idref{12}: tensors of rank-$3$ and border rank-$3$}\label{ssec:orbit_7_8_11_12}
Accordingly to the table in \Cref{ssec:finite_orbits}, we have two normal forms for concise tensors in $\bbC^2 \otimes \bbC^2 \otimes \bbC^3$ of rank three. These are
\begin{align*}
    T_7 & = \bfe^1_1 \otimes (\bfe_1^2 \otimes \bfe_1^3 +\bfe^2_2\otimes \bfe_3^3)+\bfe^1_2\otimes \bfe^2_1\otimes \bfe^3_2, \\
    T_8 & = \bfe_1^1 \otimes (\bfe^2_1\otimes \bfe^3_1+\bfe^2_2\otimes \bfe^3_2)+\bfe^1_2\otimes (\bfe^2_1\otimes \bfe^3_2+\bfe^2_2\otimes \bfe^3_3).
\end{align*}
The orbits n.\idref{11} and n.\idref{12} correspond to the orbits n.\idref{7} and n.\idref{8} by swapping second and third factor. 

Both of them are non-identifiable rank-3 tensors and they have been classified in \cite[Theorem 7.1]{BBS} 
(see also \cite[Corollaries 3.2 and 3.3]{santarsiero2023algorithm}). In particular \cite[Examples 3.6 and 3.7]{BBS} show a geometric description of  the cases $T_7$ and $T_8$ respectively: both cases are generic elements in a certain $\langle \nu_{223}(\calC \times \mathbb P^2) \rangle = \sigma_3(\nu_{223}(\calC \times \mathbb P^2))$ with $\calC$ a conic section of $X_{22}$ which is irreducible, in the case $T_8$, and reducible for $T_7$.
This description together with the fact that rank-2 tensors in $\bbC^2 \otimes \bbC^2 \otimes \bbC^3$ are non concise along the third factor and do not belong to $\tau(X_{223})$ lead to the following.

\begin{proposition}\label{prop:T7-8}
Let $T\in \calO_i\in \CC^2_1\otimes \CC^2_2\otimes \CC^3$, $i=7,8$, as in \Cref{ssec:finite_orbits}. Then $[T]\in \langle \nu_{233}(\mathcal{C}\times \mathbb{P}^2) \rangle $ with $\mathcal{C}$ a conic section of $X_{22}\subset \PP(\CC^2_1\otimes \CC^2_2)$ and
$$
 \calD_T =\  \{P=\bfa\otimes \bfb \otimes \bfc ~:~ C([\bfa\otimes \bfb])=0, \exists\, \lambda \in \mathbb C \hbox{ s.t. either } H_{223}(T-\lambda P) \neq 0 \hbox{ or } T-\lambda P \in \Sub_{122}\cup \Sub_{212} \},
$$ 
with $C$ the $(1,1)$-form of $(\CC^{2}_1\otimes \CC^{2}_2)^\vee$ defining $\mathcal{C}$.
\end{proposition}

\begin{remark}
For the particular case of $T\in \calO_7$ we also have a more elementary description given in terms of the two matrices describing it. More precisely, notice that by the $ \GL_2\times \GL_2\times \GL_3$-action, any $T\in \calO_7$ is
$$
T=\bfu_1\otimes \bfu_2\otimes \bfw+\bfu_1\otimes \bfu_2'\otimes \bfw'+\bfu_1'\otimes \bfu_2\otimes \bfw''.
$$
Call $A_1=\bfu_2\otimes \bfw+\bfu'_2\otimes \bfw'$, $A_2=\bfu_1\otimes \bfw+\bfu_1'\otimes \bfw''$. Then,
\begin{equation}\label{eq: case7 description 1}
    T=\bfu_1\otimes A_1+\bfu_1'\otimes \bfu_2\otimes \bfw''
\end{equation}
and, {if $\iota_{12} : \bbC^2_1 \otimes\bbC^2_2 \otimes \bbC^3 \to \bbC^2_2 \otimes\bbC^2_1 \otimes \bbC^3$ swaps the first two factors, then }
\begin{equation}\label{eq: case7 description 2}
    \iota_{12}(T) = \bfu_2\otimes A_2+\bfu_2'\otimes \bfu_1\otimes \bfw'.
\end{equation}
\Cref{eq: case7 description 1} and \cref{eq: case7 description 2} clearly show that {$ \{\bfu_1\} \otimes \calD_{A_1}$ and $\iota_{12}(\{\bfu_2\} \otimes \calD_{A_2})$ are contained in} $\calD_T$ and, by \cite[Proposition 4.6]{BBS:appendix}, {their union 
is actually equal to} $\calD_T$. {In particular,} this proves the following.
$$ 
    \calD_T=\left\{ \bfv^1 \otimes \bfv^2 \otimes \bfv^3 ~:~ \substack{\bfv^2 \otimes \bfv^3 \in \calD_{A_1} \text{ and } \bfv^1 = \bfu^1, \text{ or } \\ \bfv^1 \otimes \bfv^3 \in \calD_{A_2} \text{ and } \bfv^2 = \bfu^2}\right\}. 
$$
\end{remark}

\subsection{Orbit n.\protect\idref{9}: tensors of maximal rank}\label{ssec:orbit_9} 
As recalled in \Cref{ssec:finite_orbits}, the normal form of rank-4 concise tensors in $\bbC^2 \otimes \bbC^2 \otimes \bbC^n$ is given by 
\[
    T_9 = \bfe^1_1\otimes(\bfe^2_1 \otimes \bfe^3_1 + \bfe_2^2\otimes\bfe_3^3) + \bfe_2^1\otimes(\bfe_1^2\otimes\bfe_2^3 + \bfe_2^2 \otimes \bfe_4^3).
\]

\begin{notation*}
    Given a tensor $T \in \bbC^{n_1}\otimes\cdots\otimes\bbC^{n_k}$, we denote by $T^* \in (\bbC^{n_1}\otimes\cdots\otimes\bbC^{n_k})^\vee$ the tensor in dual coordinates. We write $\langle \cdot,\cdot \rangle$ for the dual pairing $(\bbC^{n_1}\otimes\cdots\otimes\bbC^{n_k})^\vee \times (\bbC^{n_1}\otimes\cdots\otimes\bbC^{n_k}) \rightarrow \bbC$. {Given a tensor $T^* \in (\bbC^{n_1}\otimes\cdots\otimes\bbC^{n_k})^\vee$ we denote by $V(T^*)$ its vanishing set, i.e., the kernel of $\langle T^*,\cdot \rangle$.}
\end{notation*}
\begin{proposition}\label{prop:orbit9}
    Let $T \in \calO_9 \subset \bbC^2 \otimes \bbC^2 \otimes \bbC^4$ such that $(A, B, C) \cdot T = T_9$. Then, {$\calF_T = V((A^T, B^T, C^T) \cdot T^*)$}. 
\end{proposition}

\begin{proof}
    Since $T\in (\bbC^2 \otimes \bbC^2 \otimes \bbC^4) \smallsetminus \sigma_3(X_{224})$, it has rank 4, then 
    \[
        \mathcal{F}_T = \{P \in X_{224} ~:~ [T-\lambda P] \not\in \sigma_3(X_{224}) ~ \forall \lambda \in \bbC\}.
    \]
    Assume first $T = T_9$. In this case, the determinant of the third flattening of $T-\lambda P$ for $P = \bfa \otimes \bfb \otimes \bfc$ is 
    \[          \lambda(a_1b_1c_1+a_2b_1c_2+a_1b_2c_3+a_2b_2c_4)-1 = \lambda \langle T_9^*,\bfa \otimes \bfb \otimes \bfc \rangle - 1.
    \]
    Now, let $T \in \calO_9$ such that $(A \otimes B \otimes C) \cdot T = T_9$. Then, {as observed in \Cref{rmk:change_coordinates}}, $\bfa \otimes \bfb \otimes \bfc \in \calF_T$ if and only if $(A \otimes B \otimes C) \cdot (\bfa \otimes \bfb \otimes \bfc) \in \calF_{T_9}$. Hence, the conclusion follows from the computation of $\calF_{T_9}$ and 
    \[  
        \langle T_9^*,(A \otimes B \otimes C) \cdot (\bfa \otimes \bfb \otimes \bfc) \rangle = \langle (A^T \otimes B^T \otimes C^T) \cdot T_9^* , \bfa\otimes \bfb \otimes \bfc\rangle. \qedhere
    \] 
\end{proof}

\section{Tensors in $\bbC^2 \otimes \bbC^3 \otimes \bbC^n$}\label{sec:23n}

In this section, we explain how to compute the decomposition locus of tensors whose normal form is concise in $\bbC^2 \otimes \bbC^3 \otimes \bbC^n$, for $n \geq 3$. 

\subsection{Orbits n.\protect\idref{14} and n.\protect\idref{18}: concise rank-$3$ tensors of border rank-$3$} \label{ssec:23n_rank3_border3}
As recalled in \Cref{table:kronecker}, a rank-3 tensor $T\in \mathbb{C}^2\otimes \mathbb{C}^3\otimes \mathbb{C}^3$ of border rank 3 belongs either to the orbit n.\idref{14} or to the orbit n.\idref{18} whose normal forms are: 
\begin{align*}
T_{14}= &\ \bf{e}^1_1\otimes( \bf{e}^2_1\otimes \bf{e}^3_1 +\otimes \bf{e}_2^2\otimes \bf{e}_2^3)+ \bf{e}_2^1 \otimes \bf{e}_3^2\otimes \bf{e}_3^3,\\
T_{18}=& \ \bf{e}^1_1\otimes \bf{e}^2_1\otimes \bf{e}^3_1 + \bf{e}^1_2\otimes \bf{e}^2_3\otimes \bf{e}^3_3 + (\bf{e}^1_1+\bf{e}^1_2)\otimes \bf{e}^2_2\otimes \bf{e}^3_2.
\end{align*}

\begin{remark}\label{remark:T14}
Case \idref{14} 
can be written as $T_{14}=\bf{e}_1^1\otimes I_{12}+ \bf{e}_2^1 \otimes \bf{e}_3^2\otimes \bf{e}_3^3$
    where $I_{12}$ is a $3\times 3$ matrix having ones in position $(1,1), (2,2)$ and zero elsewhere. 
        
   By  \cite[Proposition 3.10, Item 2.]{BBS} (cf. also \cite[Lemma 4.4]{BBS:appendix}), we have that 
    $$\mathcal{D}_{T_{14}}= (\alpha \bf{e}_1^1)\otimes\mathcal{D}_{I_{1,2}}\cup \{(\alpha_1 \bf{e}_2^1) \otimes (\alpha_2 \bf{e}_3^2)\otimes (\alpha_3 \bf{e}_3^3)\}$$
    with $\alpha,\alpha_1, \alpha_2, \alpha_3\in \mathbb{C}$.
    Further, by the $\GL_2\times \GL_3\times \GL_3$-action, any $T\in \calO_{14}$ can be expressed as $T=\bfu_1 \otimes M +\bfu_2\otimes \bfv \otimes \bfw$, for independent $\bfu_1,\bfu_2\in \CC^2$ and for some rank-2 square matrix $M$ of size 3 s.t. $\rk(M+ \bfv \otimes \bfw)=3$. Therefore, the decomposition locus of $T_{14}\in \calO_{14}$ is $$\calD_{T_{14}}= \{ \bfu_1\}\otimes \calD_{M}\cup \{ \bfu_2\otimes \bfv \otimes \bfw \} .$$
\end{remark}

\begin{remark} Case \idref{18} 
can be written as a rank-3 tensor as $$T_{18}=\bf{e}^1_1\otimes \bf{e}^2_1\otimes \bf{e}^3_1 + \bf{e}^1_2\otimes \bf{e}^2_3\otimes \bf{e}^3_3 + (\bf{e}^1_1+\bf{e}^1_2)\otimes \bf{e}^2_2\otimes \bf{e}^3_2.$$
This is identifiable since its $\GL_2 \times \GL_3 \times \GL_3$ orbit closure coincides with $\sigma_3(X_{233})$ which fills perfectly the ambient space and the secant degree is 1 (it is classically known that generic $(2\times k \times k )$-tensors are identifiabile, cf. e.g. \cite{HOOS}). In our particular instance, 
$$\mathcal{D}_{T_{18}}=\{\bf{e}^1_1\otimes \bf{e}^2_1\otimes \bf{e}^3_1,\; \bf{e}^1_2\otimes \bf{e}^2_3\otimes \bf{e}^3_3,\; (\bf{e}^1_1+\bf{e}^1_2)\otimes \bf{e}^2_2\otimes \bf{e}^3_2\}$$
and more generally any other tensor in the $\GL_2 \times \GL_3 \times \GL_3$ orbit of $T_{18}$ is identifiable and so its decomposition locus is given only by the 3 points arising in its unique decomposition.
\end{remark}

\subsection{Orbits n.\protect\idref{13}, n.\protect\idref{15}, n.\protect\idref{16} and n.\protect\idref{17}: rank-$4$ tensors in $\bbC^2 \otimes \bbC^3 \otimes \bbC^3$}\label{ssec:rank4_233}
As mentioned, in $\bbC^2 \otimes \bbC^3 \otimes \bbC^3$, the generic rank is $3$, but the maximal rank is $4$. The closure of the set of tensors of rank $4$ is the dual variety $\calX_{233}^\vee$ of the Segre variety $X_{233}$. This is classically known to be a hypersurface in $\mathbb{P}^{17}$ (cf. \cite[Ch. 14, Theorem 1.3]{GKZ}) whose defining equation is the previously defined hyperdeterminant. Following the Schl{\"a}fli method (cf. \cite[Ch. 14, Sec.~4]{GKZ}), such hyperdeterminant is computed as follows. Regard the generic tensor in $\bbC^2 \otimes\bbC^3 \otimes \bbC^3$ as a pencil of $3 \times 3$ matrices, i.e., write $T \in \sum_{i,j,k} t_{i,j,k} \bfe^1_i \otimes \bfe^2_j \otimes \bfe^3_k \in \bbC^2 \otimes \bbC^3 \otimes \bbC^3$ as 
\[
    T = u\left(\begin{smallmatrix}
        t_{1,1,1} & t_{1,1,2} & t_{1,1,3} \\
        t_{1,2,1} & t_{1,2,2} & t_{1,2,3} \\
        t_{1,3,1} & t_{1,3,2} & t_{1,3,3} \\
    \end{smallmatrix}\right) + v
    \left(\begin{smallmatrix}
        t_{2,1,1} & t_{2,1,2} & t_{2,1,3} \\
        t_{2,2,1} & t_{2,2,2} & t_{2,2,3} \\
        t_{2,3,1} & t_{2,3,2} & t_{2,3,3} \\
    \end{smallmatrix}\right) =: uA_0 + vA_1.
\]
Its hyperdeterminant is equal to the discriminant of the binary cubic polynomial in $\bbC[t_{i,j,k}][u,v]$ given by the determinant of the matrix $uA_0 + vA_1$ (cf. \cite[Ch. 14, Sec.~4]{GKZ}).

Following the classification in  \cite[Table 3]{BL}, the following are the normal forms for the orbits of concise tensors of rank $4$ contained in $\bbC^2 \otimes \bbC^3 \otimes \bbC^3$:
\begin{equation}
\begin{aligned}
T_{13} & =\bf{e}_1^1\otimes (\bf{e}_1^2\otimes \bf{e}_1^3 + \bf{e}_2^2 \otimes \bf{e}_3^3 ) + \bf{e}_2^1\otimes ( \bf{e}_1^2\otimes \bf{e}_2^3+\bf{e}_3^2\otimes \bf{e}_3^3), \\
T_{15} & =\bf{e}_1^1\otimes (\bf{e}_1^2\otimes \bf{e}_1^3+ \bf{e}_2^2\otimes \bf{e}_2^3+\bf{e}_3^2\otimes \bf{e}_3^3 ) + \bf{e}_2^1\otimes  \bf{e}_1^2\otimes \bf{e}_2^3,\\
T_{16} & =\bf{e}_1^1\otimes (\bf{e}_1^2\otimes \bf{e}_1^3+ \bf{e}_2^2\otimes \bf{e}_2^3+ \bf{e}_3^2\otimes \bf{e}_3^3 ) + \bf{e}_2^1\otimes ( \bf{e}_1^2\otimes \bf{e}_2^3+\bf{e}_2^2\otimes \bf{e}_3^2  ),\\
T_{17} & =\bf{e}_1^1\otimes (\bf{e}_2^2\otimes \bf{e}_1^3 + \bf{e}_2^2\otimes \bf{e}_2^3 ) + \bf{e}_2^1\otimes ( \bf{e}_2^2\otimes \bf{e}_2^3+\bf{e}_3^2\otimes \bf{e}_3^3). 
\end{aligned}
\label{eq:rank4_orbits}
\end{equation}
The dual variety $\calX_{233}^\vee$ is the Zariski-closure of the orbit of $T_{17}$.
The concise tensors in the singular locus of $\calX_{233}^\vee$ correspond to the normal forms $T_{13}, T_{15}, T_{16}$ and also
\begin{equation}\label{eq:orbit14}
    T_{14} = \bfe_1^1\otimes (\bfe_1^2\otimes \bfe_1^3 + \bfe_2^2 \otimes \bfe_2^3) + \bfe_2^1 \otimes \bfe_3^2 \otimes \bfe_3^3.
\end{equation}
The latter one, $T_{14}$, has rank $3$ instead of $4$, see \cite[page 108]{parfenov2001orbits}.

We describe here a procedure to compute the forbidden locus of a concise tensor $T$ of rank $4$ in $\bbC^2\otimes\bbC^3\otimes\bbC^3$, i.e., a tensor $T$ belonging to one of the orbits of the tensors in \cref{eq:rank4_orbits}.  By definition, we look for the rank-one tensors $P$ such that $\rk(T-\lambda P) = 4$ for all $\lambda$'s. By \cite{BL}, the latter is guaranteed by the following:
\begin{enumerate}
    \item\label{condition1} $T-\lambda P$ is concise for every $\lambda$: indeed, in $\bbC^1 \otimes \bbC^3 \otimes \bbC^3$ and $\bbC^2 \otimes \bbC^2 \otimes \bbC^3$ the maximal rank is $3$, this implies that $T- \lambda p \not\in \sigma_2(X_{233}) \cup 
    {\rm Sub}_{223} \cup {\rm Sub}_{232}
    $;
    \item\label{condition2} $T-\lambda P \in \calX_{233}^\vee$ for every $\lambda$: indeed, a necessary condition for having rank $4$ is to be on such  a variety;
    \item\label{condition3} $T-\lambda P \not\in (\GL_2\times\GL_3\times \GL_3)\cdot T_{14}$ whose closure is the join variety $J(X_{233},\mathbb{P}^1 \times \sigma_2(\mathbb{P}^2 \times \mathbb{P}^2))$ (this is very clear from the generator $T_{14}$ but see also \cite[Table 4, Orbit $\overline{\mathcal{O}}_{XIV}$]{HLT}) for every $\lambda$'s: indeed, by \cite[Table 3]{BL}, that is the only orbit of concise tensors of rank $3$ contained in the $\calX_{233}^\vee$.
\end{enumerate} 
The first two conditions are easy to check. Condition \eqref{condition1} is checked by imposing the maximal rank of the flattenings. Condition \eqref{condition2} is checked by imposing the vanishing of the hyperdeterminant. As regards condition \eqref{condition3}, we need a few more observations on the orbits of tensors in \cref{eq:rank4_orbits,eq:orbit14}. 

This description is summarized in the following. 

\begin{proposition}\label{prop:rank4_233}
For any tensor $T \in \calO_{13}\cup\calO_{15}\cup\calO_{16}\cup\calO_{17}$, the forbidden locus $\mathcal{F}_T$ is given by the points $P \in \mathbb{C}^2 \otimes \mathbb{
C}^3 \otimes \mathbb{C}^3$ such that, for all $\lambda \in \mathbb{C}$, $[T- \lambda P]$ belongs to $$\calX_{233}^\vee \smallsetminus \left(\sigma_2(X_{233}) \cup \Sub_{223} \cup \Sub_{232} \cup J^\circ(X_{233}, \mathbb{P}^1 \times \sigma_2(X_{3,3})))\right),$$
where $J^\circ$ denotes the dense subset of the join variety of points lying on lines joining distinct points.
\end{proposition}

In order to make this description effective, we need the following lemma which is straightforward and should be classically known: we include here a proof for sake of completeness.

\begin{lemma}\label{lemma:rank4_characterization}
    Let $T \in \bbC^2\otimes\bbC^3\otimes\bbC^3$ be a concise tensor in $\calX_{233}^\vee$. Let $\ell_T \subset \PP (\bbC^3\otimes\bbC^3)$ be the line corresponding to the pencil of $3\times 3$ matrices associated to $T$. Denote $G = \GL_2 \times \GL_3 \times \GL_3$. Then:
    \begin{enumerate}
        \item if $T \in G \cdot T_{13}$: $\ell_T \subset \sigma_2(X_{33})$;
        \item if $T \in G \cdot T_{15}$: the support of $\ell_T \cap \sigma_2(X_{33})$ is a single point of rank $1$;
        \item if $T \in G \cdot T_{16}$: the support of $\ell_T \cap \sigma_2(X_{33})$ is a single point of rank $2$;
        \item if $T \in G \cdot T_{14}$: the support of the intersection $\ell_T \cap \sigma_2(X_{33})$ are two distinct points and $\ell_T \cap X_{33} \neq \emptyset$;
        \item if $T \in G \cdot T_{17}$: the support of the intersection $\ell_T \cap \sigma_2(X_{33})$ are two distinct points and $\ell_T \cap X_{33} = \emptyset$.
    \end{enumerate}
\end{lemma}
\begin{proof}
    Since the Segre variety $X_{33}$ and its secant varieties are closed under the action of $\GL_3 \times \GL_3$, it is enough to prove the claim for $T \in \{T_{13}, T_{14}, T_{15}, T_{16}, T_{17}\}$.
    
    The pencils corresponding to $T_{13}, \ldots , T_{17}$ are associated to the following matrices depending on $(u:v) \in \bbP^1$:
     \[
        \ell_{T_{13}} : \left(\begin{smallmatrix}
           u & v & 0 \\
           0 & 0 & u \\
           0 & 0 & v
        \end{smallmatrix}\right);
        \quad 
         \ell_{T_{14}}: \left(\begin{smallmatrix}
           u & 0 & 0 \\
           0 & u & 0 \\
           0 & 0 & v
        \end{smallmatrix}\right),
        \quad
         \ell_{T_{15}}: \left(\begin{smallmatrix}
           u & v & 0 \\
           0 & u & 0 \\
           0 & 0 & u
        \end{smallmatrix}\right),
        \quad
         \ell_{T_{16}}: \left(\begin{smallmatrix}
           u & v & 0 \\
           0 & u & v \\
           0 & 0 & u
        \end{smallmatrix}\right),
        \quad 
        \ell_{T_{17}}: \left(\begin{smallmatrix}
            u & v & 0 \\
            0 & u & 0 \\
            0 & 0 & v
        \end{smallmatrix}\right).
     \]
     Then:
     \begin{enumerate}
         \item the determinant of the matrix associated to $\ell_{T_{13}}$ is identically $0$ and then $\ell_{T_{13}} \subset \sigma_2(X_{33})$.
         \item[(2)-(3)] The determinant of the matrices associated to $\ell_{T_{15}}$ and $\ell_{T_{16}}$ is $u^3$, i.e., the pencils $\ell_{T_{15}}$ and $\ell_{T_{16}}$ intersect $\sigma_2(X_{33})$ only for $(u:v) = (0:1)$. In the first case it corresponds to a rank-one matrix, in the second case it corresponds to a rank-two matrix.
         \item[(4)-(5)] The determinant of the matrices associated to $\ell_{T_{14}}$ and $\ell_{T_{17}}$ is $u^2v$, i.e., the pencils $\ell_{T_{14}}$ and $\ell_{T_{17}}$ intersect $\sigma_2(X_{33})$ for $(u:v) \in \{(0:1),(1:0)\}$. However, the pencil $\ell_{T_{14}}$ contains a point of rank-one for $(u:v) = (0:1)$ while the ideal of $2 \times 2$ minors of the matrix associated to $\ell_{T_{17}}$ is the irrelevant ideal $(u,v)$, i.e., $\ell_{T_{17}} \cap X_{3,3} = \emptyset$. \qedhere
    \end{enumerate}
\end{proof}

\begin{algorithm}[htb]
\caption{Test if a rank-$1$ tensor lies in the forbidden locus of a concise rank-$4$ tensor of $\bbC^2 \otimes \bbC^3 \otimes \bbC^3$}
\flushleft{
\textbf{Input:} }
\begin{itemize}
    \item Rank-$4$ concise tensor $T \in \bbC^2 \otimes \bbC^3 \otimes \bbC^3$.
    \item Rank-one tensor $P \in \bbC^2 \otimes \bbC^3 \otimes \bbC^3$.
\end{itemize}
\flushleft{
\textbf{Output:} Whether $P \in \calF_T$ or $P \not\in \calF_T$.
}

\flushleft{
\textbf{Procedure:}
}
\begin{enumerate}
    \item[{\it Step 1.}] Compute the principal ideal $I := (f) \in \bbC[\lambda]$ of maximal minors of all flattenings of $T-\lambda P$.
    \begin{itemize}
        \item If $I \neq (1)$ then $T-\lambda P$ is not concise for all roots of $f$, \\ hence \textbf{return} $P \not\in \calF_T$;
        \item otherwise, if $I = (1)$ then $T-\lambda P$ is concise for all $\lambda$ and we proceed to \textit{Step 2}.
    \end{itemize}
    \item[{\it Step 2.}]\label{(2)} Compute the hyperdeterminant of $T-\lambda P$, i.e., $H_{233}(T-\lambda P) \in \bbC[\lambda]$. 
        \begin{itemize}
        \item If $H_{233}(T-\lambda p)\not\equiv 0$ then $T-\lambda P \not\in \calX_{233}^\vee$ for all $\lambda$'s different from the roots, \\ hence \textbf{return} $P \not\in \calF_T$;
        \item otherwise, $T-\lambda P \subset X_{233}^\vee$ for all $\lambda$'s and we proceed to \textit{Step 3}.
    \end{itemize}
    \item[{\it Step 3.}] Arrived at this point, we know that $T- \lambda P \subset \calX_{233}^\vee$ and all its points are concise. We need to distinguish if some of them belongs to the orbit $\calO_{14}$ or not. We apply \Cref{lemma:rank4_characterization}.

    Let $M_{T-\lambda P} \in {\rm Mat_{3 \times 3}}(\bbC[u,v][\lambda])$ be the pencil of matrices corresponding to $T-\lambda P$ for any $\lambda$. \\ We compute the determinant $\det(T-\lambda P) := \det(M_{T-\lambda P}) \in \bbC[u,v][\lambda]$ and the ideal of $2 \times 2$ minors $I_2(T-\lambda P) := I_2(M_{T-\lambda P}) \subset \bbC[u,v][\lambda]$.

    \begin{itemize}
        \item If, for some $\lambda \neq 0$, we have that $I_2(T-\lambda P) \neq (1)$ and $\det(T-\lambda P)$ is neither zero or a pure cube, then $P \not\in\calF_T$;
        \item otherwise, \textbf{return} $P \in \calF_T$.
    \end{itemize}
  \end{enumerate}
\label{algo:233}
\end{algorithm}

We show here how to use \Cref{algo:233} to give a symbolic description of the forbidden loci for tensors in the orbits $\calO_{13},\calO_{15},\calO_{16},\calO_{17}$, respectively.
{\begin{notation*}
    Given multihomogeneous polynomials $f_1,\ldots,f_m \in \bbC[\bfa]\otimes\bbC[\bfb]\otimes\bbC[\bfc]$, we denote by $Z(f_1,\ldots,f_m)$ the zero-locus of $f_i$'s as a subspace of the set of rank-one tensors. 
\end{notation*}}
\subsubsection{Orbit n.\protect\idref{13}}\label{ssec:n13} Consider the generator of the orbit 
\[
    T_{13} = \bfe^1_1\otimes(\bfe_1^2\otimes \bfe_1^3+\bfe_2^2\otimes \bfe_3^3)+\bfe_2^1\otimes (\bfe_1^2\otimes \bfe_2^3+\bfe_3^2 \otimes \bfe_3^3). 
\]
Let $P = \bfa \otimes \bfb \otimes \bfc$.

{\textit{Step 1.}} The ideal of maximal minors of all flattenings of $T_{13} - \lambda P$ is
\[
    I = \left(b_3c_2, b_2c_2, b_3c_1, b_2c_1,\lambda(a_1b_1c_1+a_2b_1c_2+a_1b_2c_3+a_2b_3c_3)-1\right) \subset \bbC[\bfa,\bfb,\bfc][\lambda].
\]
Note that the latter generator is exactly $\lambda \langle T_{13}^*,P \rangle -1$ while 
\[
    (b_3c_2, b_2c_2, b_3c_1, b_2c_1) = (b_2,b_3) \cap (c_1,c_2).
\]
A necessary condition for $P \in \calF_{T_{13}}$, is $I = (1) \subset \bbC[\lambda]$, then 
\begin{equation}\label{eq:orbit13_firstinclusion}
    \calF_{T_{13}} \subset  \left[\bbC^2 \otimes \bbC^3 \otimes \bbC^3 \smallsetminus \left (Z(b_2,b_3) \cup Z(c_1,c_2)\right)\right] \cup V(T_{13}^*).
\end{equation}

{\textit{Step 2.}} Another necessary condition $P \in \calF_{T_{13}}$ is that $T_{13} - \lambda P$, interpreted as a pencil in $\bbC^3\otimes\bbC^3$ is contained in the dual variety $\calX^\vee_{33}$ which is defined by the vanishing of the hyperdeterminant. As already recalled, it is defined by the discriminant of the cubic polynomial given by the determinant of $T_{13} - \lambda P$:
\begin{align*}
    H_{233}(T_{13}-\lambda P) = -\lambda ^{4}\left(b_{3}c_{1}-b_{2}c_{2}\right)^{2}\left(a_{1}c_{1}+a_{2}c_{2}\right)^{2}\left(a_{1}b_{2}+a_{2}b_{3}\right)^{2}.
\end{align*}
Hence,
\begin{align*}\label{eq:orbit13_secondinclusion}
    \calF_{T_{13}} \subset  & \left[\left(Z(b_3c_1-b_2c_2) \cup Z(a_1c_1+a_2c_2) \cup Z(a_1b_2+a_2b_3)\right) \smallsetminus \left (Z(b_2,b_3) \cup Z(c_1,c_2)\right)\right] \\
   &  \cup \left[\left(Z(b_3c_1-b_2c_2) \cup Z(a_1c_1+a_2c_2) \cup Z(a_1b_2+a_2b_3)\right) \cap V(T_{13}^*) \right].
\end{align*}

{\textit{Step 3.}} Now, we consider $T-\lambda P$ as a pencil of $3 \times 3$ matrices: its determinant is
\[
\det(T_{13}  - \lambda P) = \lambda \left(c_{2}u-c_{1}v\right)\left(b_{3}u-b_{2}v\right)\left(a_{1}u+a_{2}v\right).
\]
This is identically a pure cube if and only if $P \in Z(b_3c_1-b_2c_2) \cap Z(a_1c_1+a_2c_2) \cap Z(a_1b_2+a_2b_3)$. In such a case, we know that $P \in \calF_T$ by \Cref{lemma:rank4_characterization}.

In order to understand the remaining cases, we study the intersection of the pencil with the Segre $X_{33}$: the radical ideal of $2\times 2$ minors decomposes as 
\[
(u,v) \cap \left(b_{3}c_{1}-b_{2}c_{2},\,c_{2}u-c_{1}v,\,b_{3}u-b_{2}v, \lambda(a_{1}b_{1}c_{1}+a_{2}b_{1}c_{2}+a_{1}b_{2}c_{3}+a_{2}b_{3}c_{3})-1\right);
 \]
in other words, if $P \in Z(b_3c_1 - b_2c_2) \smallsetminus V(T_{13}^*)$, then $T- \lambda P \in \calO_{14}$ for some $\lambda$ and $P \not\in \calF_{T_{13}}$. In all the other cases, the pencil does not meet $X_{33}$ and $P \in \calF_{T_{13}}$. {Hence, we get the following.
\begin{corollary}  
    The forbidden locus of the tensor $T_{13} \in \bbC^2 \otimes \bbC^3 \otimes \bbC^3$ defined in \Cref{table:kronecker} is
    \begin{align*}
    \calF_{T_{13}} =  & \left[\left(Z(a_1c_1+a_2c_2) \cup Z(a_1b_2+a_2b_3)\right) \smallsetminus \left (Z(b_2,b_3) \cup Z(c_1,c_2)\right)\right] \\
   &  \cup \left[\left(Z(b_3c_1-b_2c_2) \cup Z(a_1c_1+a_2c_2) \cup Z(a_1b_2+a_2b_3)\right) \cap V(T_{13}^*) \right].
    \end{align*}
\end{corollary}
}
\subsubsection{Orbit n.\protect\idref{15}}\label{ssec:n15} Consider the generator of the orbit
\[
T_{15} = \bfe_1^1\otimes (\bfe_1^2\otimes \bfe_1^3+\bfe_2^2\otimes \bfe_2^3+\bfe_3^2\otimes \bfe_3^3)+\bfe_2^1\otimes \bfe_1^2\otimes \bfe_2^3.
\]
Let $P = \bfa \otimes \bfb \otimes \bfc$. 

{\it Step 1.} The ideal of maximal minors of all flattenings of $T_{15}-\lambda P$ is
\[
    \left(a_{2}c_{3},\,b_{2}c_{1},\,a_{2}c_{1},\,a_{2}b_{3},\,a_{2}b_{2},\lambda(a_{1}b_{1}c_{1}+a_{2}b_{1}c_{2}+a_{1}b_{2}c_{2}+a_{1}b_{3}c_{3})-1\right) \subset \bbC[\bfa,\bfb,\bfc].
\]
Note that the latter generator is exactly $\lambda \langle T_{15}^*,P\rangle -1$ while \[
(a_{2}c_{3},\,b_{2}c_{1},\,a_{2}c_{1},\,a_{2}b_{3},\,a_{2}b_{2}) = (b_2,b_3,c_1,{c_3}) \cap (a_2,b_2) \cap (a_2,{c_1}).\]
Hence, a necessary condition for $P \in \calF_{T_{15}}$ is $I = (1) \subset \bbC[\lambda]$, then 
\begin{equation}\label{eq:orbit15_firstinclusion}
    \calF_{T_{15}} \subset  \left[\bbC^2 \otimes \bbC^3 \otimes \bbC^3 \smallsetminus \left (Z(b_2,b_3,c_1,{c_3}) \cup Z(a_2,b_2) \cup Z(a_2,{c_1})\right)\right] \cup V(T_{15}^*).
\end{equation}

{\it Step 2.} Another necessary condition for $P \in \calF_{T_{15}}$ is that the line $T_{15}-\lambda P$ is all contained in the dual variety $\calX_{33}^\vee$. Restricting the hyperdeterminant at such pencil, we get:
\begin{equation}\label{eq:orbit15_hyper}
    H_{233}(T_{15}-\lambda P) = a_2^2b_2^2c_1^2\lambda^3(D\cdot\lambda - 4a_2b_2c_1) \in \CC[\bfa,\bfb,\bfc][\lambda]
 \end{equation}
 where 
 \begin{align*}
     D = a_{2}^{2}b_{1}^{2}c_{1}^{2}&+2a_{1}a_{2}b_{1}b_{2}c_{1}^{2}+a_{1}^{2}b_{2}^{2}c_{1}^{2}+2a_{2}^{2}b_{1}b_{2}c_{1}c_{2}+2a_{1}a_{2}b_{2}^{2}c_{1}c_{2}\\
     &+a_{2}^{2}b_{2}^{2}c_{2}^{2}+2a_{2}^{2}b_{1}b_{3}c_{1}c_{3}+2a_{1}a_{2}b_{2}b_{3}c_{1}c_{3}+2\,a_{2}^{2}b_{2}b_{3}c_{2}c_{3}+a_{2}^{2}b_{3}^{2}c_{3}^{2}.
\end{align*}
The univariate polynomial in \cref{eq:orbit15_hyper} is identically zero if and only if $a_2b_2c_1 = 0$. Hence, we refine \cref{eq:orbit15_firstinclusion} as 
\begin{align}\label{eq:orbit15_secondinclusion}
    \calF_{T_{15}} \subset & \left[ Z(a_2b_2c_1) \smallsetminus (Z(b_2,b_3,c_1,{c_3}) \cup Z(a_2,b_2{c_1}))\right] \nonumber \\
    & \cup \left[ Z(a_2b_2c_1) \cap V(T_{15}^*) \right].
\end{align}

{\it Step 3.} Now, we apply \Cref{lemma:rank4_characterization}. We know that $P \not\in \calF_{T_{15}}$ if $T_{15}-\lambda P \in G \cdot T_{14}$ for some $\lambda$. A necessary condition, is that the pencil associated to $T_{15}-\lambda P$ contains rank-$1$ matrices for some $\lambda$. The radical ideal of $2\times 2$ minors of $T_{15}-\lambda P$, saturated with respect to the ideal $(u,v) $, is 
\[
(u, \lambda a_2b_3c_3, \lambda a_2b_2c_3, \lambda a_2b_3c_1) = (a_2,u) \cap (b_2,\lambda,u) \cap (b_3,b_2,u) \cap (c_1,\lambda,u) \cap (c_3,c_1,u).
\]
Since, we only need to consider $\lambda \neq 0$, we have only three cases in which, to fully apply \Cref{lemma:rank4_characterization}, we need to check also the intersection of the pencil with $\sigma_2(X_{33})$. Indeed, if $a_2 = 0$ or $b_2 = b_3 = 0$ or $c_1 = c_3 = 0$, then, for every $\lambda$, we have that the point $(0:1)$ of the pencil has rank equal to $1$. 

The determinant of the pencil $T_{15}- \lambda P$ becomes:
\begin{itemize}
    \item for $a_2 = 0$: 
    \[
     \det (T_{15}-\lambda P) = \left(-\lambda(\,a_{1}b_{1}c_{1}+\,a_{1}b_{2}c_{2}+\,a_{1}b_{3}c_{3})+1\right)u^{3}+\lambda\,a_{1}b_{2}c_{1}u^{2}v
          \]
    which is never identically zero, but it is identically a pure cube if $a_1b_2c_1 = 0$;
    \item for $b_2 = b_3 = 0$: 
    \[
    \det (T_{15}-\lambda P) = \left(-\lambda a_1b_1c_1 + 1\right)u^{3}-\lambda\,a_{2}b_{1}c_{1}u^{2}v
    \]
    which is never identically zero, but it is identically a pure cube if $a_2b_1c_1 = 0$;
    \item for $c_1 = c_3 = 0$: 
    \[
    \det (T_{15}-\lambda P) = \left(-\lambda a_1b_2c_2 + 1\right)u^{3}-\lambda\,a_{2}b_{2}c_{2}u^{2}v
    \]
    which is never identically zero, but it is identically a pure cube if $a_2b_2c_2 = 0$.
\end{itemize}
{Hence, we get the following.
\begin{corollary}  
    The forbidden locus of the tensor $T_{15} \in \bbC^2 \otimes \bbC^3 \otimes \bbC^3$ defined in \Cref{table:kronecker} is
    \begin{align*}
        \calF_{T_{15}} = & \Big[\left[ Z(a_2b_2c_1) \smallsetminus (Z(b_2,b_3,c_1,{c_3}) \cup Z(a_2,b_2{c_1}))\right] \cup \left[ Z(a_2b_2c_1) \cap V(T_{15}^*) \right]\Big] \\
        & \smallsetminus \Big[ [Z(a_2) \smallsetminus Z(a_1b_2c_1)] \cup [Z(b_2,b_3) \smallsetminus Z(a_2b_1c_1)] \cup [Z(c_1,c_3) \smallsetminus Z(a_2b_2c_2) ]\Big].
    \end{align*}
\end{corollary}
}

\subsubsection{Orbit n.\protect\idref{16}}\label{ssec:n16}
Consider the generator of the orbit
\[
T_{16} = \bfe_1^1\otimes (\bfe_1^2\otimes \bfe_1^3+\bfe_2^2\otimes \bfe_2^3+\bfe_3^2\otimes \bfe_3^3)+\bfe_2^1\otimes (\bfe_1^2\otimes \bfe_2^3+\bfe_2^2\otimes \bfe_3^3).
\]
Let $P = \bfa \otimes \bfb \otimes \bfc$. 

{\it Step 1.} The ideal of maximal minors of all flattenings of $T_{16} - \lambda P$ is
\[
    \left(a_{2},\,b_{3}c_{2},\,b_{2}c_{2},\,b_{3}c_{1},\,b_{2}c_{1}, \lambda(a_{1}b_{1}c_{1}+a_1b_2c_2+a_{1}b_{3}c_{3}+a_2b_1c_2+a_2b_2c_3)-1\right).
\]
Note that the latter generator is exactly $\lambda \langle T_{16}^*,P \rangle -1$ while 
\[
(a_{2},\,b_{3}c_{2},\,b_{2}c_{2},\,b_{3}c_{1},\,b_{2}c_{1}) = (a_2,b_2,b_3) \cap (a_2,{c_1},c_3). 
\]
Hence, 
\begin{equation}\label{eq:orbit16_firstinclusion}
    \calF_{T_{16}} \subset [\bbC^2\otimes\bbC^3\otimes\bbC^3 \smallsetminus {(Z(a_2,b_2,b_3)\cup Z(a_2,c_1,c_3))}] 
    \cup V(T_{16}^*).
\end{equation}

{\it Step 2.} Another necessary condition is that $T_{16}-\lambda P$ is contained in $\calX_{233}^\vee$ for all $\lambda$'s. Evaluating the hyperdeterminant at $T_{16}-\lambda P$ we get
\[
    H_{233}(T_{16}-\lambda P) = \lambda^2(A \lambda^2 + B \lambda + C)
 \]
 where
 {\begin{footnotesize}
 \begin{align*}
     A & = a_{2}^{4}b_{1}^{2}b_{2}^{2}c_{1}^{4}+2\,a_{1}a_{2}^{3}b_{1}b_{2}^{3}c_{1}^{4}+a_{1}^{2}a_{2}^{2}b_{2}^{4}c_{1}^{4}-4\,a_{2}^{4}b_{1}^{3}b_{3}c_{1}^{4}-8\,a_{1}a_{2}^{3}b_{1}^{2}b_{2}b_{3}c_{1}^{4}-2\,a_{1}^{2}a_{2}^{2}b_{1}b_{2}^{2}b_{3}c_{1}^{4}+2\,a_{1}^{3}a_{2}b_{2}^{3}b_{3}c_{1}^{4}\\
     &-8\,a_{1}^{2}a_{2}^{2}b_{1}^{2}b_{3}^{2}c_{1}^{4}-8\,a_{1}^{3}a_{2}b_{1}b_{2}b_{3}^{2}c_{1}^{4}+a_{1}^{4}b_{2}^{2}b_{3}^{2}c_{1}^{4}-4\,a_{1}^{4}b_{1}b_{3}^{3}c_{1}^{4}+2\,a_{2}^{4}b_{1}b_{2}^{3}c_{1}^{3}c_{2}+2\,a_{1}a_{2}^{3}b_{2}^{4}c_{1}^{3}c_{2}-10\,a_{2}^{4}b_{1}^{2}b_{2}b_{3}c_{1}^{3}c_{2}\\
     &-10\,a_{1}a_{2}^{3}b_{1}b_{2}^{2}b_{3}c_{1}^{3}c_{2}+2\,a_{1}^{2}a_{2}^{2}b_{2}^{3}b_{3}c_{1}^{3}c_{2}-8\,a_{1}a_{2}^{3}b_{1}^{2}b_{3}^{2}c_{1}^{3}c_{2}-20\,a_{1}^{2}a_{2}^{2}b_{1}b_{2}b_{3}^{2}c_{1}^{3}c_{2}-2\,a_{1}^{3}a_{2}b_{2}^{2}b_{3}^{2}c_{1}^{3}c_{2}\\
    &-8\,a_{1}^{3}a_{2}b_{1}b_{3}^{3}c_{1}^{3}c_{2}-2\,a_{1}^{4}b_{2}b_{3}^{3}c_{1}^{3}c_{2}+a_{2}^{4}b_{2}^{4}c_{1}^{2}c_{2}^{2}-8\,a_{2}^{4}b_{1}b_{2}^{2}b_{3}c_{1}^{2}c_{2}^{2}-2\,a_{1}a_{2}^{3}b_{2}^{3}b_{3}c_{1}^{2}c_{2}^{2}+a_{2}^{4}b_{1}^{2}b_{3}^{2}c_{1}^{2}c_{2}^{2}\\
    &-10\,a_{1}a_{2}^{3}b_{1}b_{2}b_{3}^{2}c_{1}^{2}c_{2}^{2}-6\,a_{1}^{2}a_{2}^{2}b_{2}^{2}b_{3}^{2}c_{1}^{2}c_{2}^{2}-2\,a_{1}^{2}a_{2}^{2}b_{1}b_{3}^{3}c_{1}^{2}c_{2}^{2}-2\,a_{1}^{3}a_{2}b_{2}b_{3}^{3}c_{1}^{2}c_{2}^{2}+a_{1}^{4}b_{3}^{4}c_{1}^{2}c_{2}^{2}-2\,a_{2}^{4}b_{2}^{3}b_{3}c_{1}c_{2}^{3}\\
     &+2\,a_{2}^{4}b_{1}b_{2}b_{3}^{2}c_{1}c_{2}^{3}-2\,a_{1}a_{2}^{3}b_{2}^{2}b_{3}^{2}c_{1}c_{2}^{3}+2\,a_{1}a_{2}^{3}b_{1}b_{3}^{3}c_{1}c_{2}^{3}+2\,a_{1}^{2}a_{2}^{2}b_{2}b_{3}^{3}c_{1}c_{2}^{3}+2\,a_{1}^{3}a_{2}b_{3}^{4}c_{1}c_{2}^{3}+a_{2}^{4}b_{2}^{2}b_{3}^{2}c_{2}^{4}\\
     &+2\,a_{1}a_{2}^{3}b_{2}b_{3}^{3}c_{2}^{4}+a_{1}^{2}a_{2}^{2}b_{3}^{4}c_{2}^{4}+2\,a_{2}^{4}b_{1}b_{2}^{2}b_{3}c_{1}^{3}c_{3}+2\,a_{1}a_{2}^{3}b_{2}^{3}b_{3}c_{1}^{3}c_{3}-12\,a_{2}^{4}b_{1}^{2}b_{3}^{2}c_{1}^{3}c_{3}-16\,a_{1}a_{2}^{3}b_{1}b_{2}b_{3}^{2}c_{1}^{3}c_{3}\\
     &-2\,a_{1}^{2}a_{2}^{2}b_{2}^{2}b_{3}^{2}c_{1}^{3}c_{3}-16\,a_{1}^{2}a_{2}^{2}b_{1}b_{3}^{3}c_{1}^{3}c_{3}-8\,a_{1}^{3}a_{2}b_{2}b_{3}^{3}c_{1}^{3}c_{3}-4\,a_{1}^{4}b_{3}^{4}c_{1}^{3}c_{3}+2\,a_{2}^{4}b_{2}^{3}b_{3}c_{1}^{2}c_{2}c_{3}-20\,a_{2}^{4}b_{1}b_{2}b_{3}^{2}c_{1}^{2}c_{2}c_{3}\\
     &-10\,a_{1}a_{2}^{3}b_{2}^{2}b_{3}^{2}c_{1}^{2}c_{2}c_{3}-16\,a_{1}a_{2}^{3}b_{1}b_{3}^{3}c_{1}^{2}c_{2}c_{3}-20\,a_{1}^{2}a_{2}^{2}b_{2}b_{3}^{3}c_{1}^{2}c_{2}c_{3}-8\,a_{1}^{3}a_{2}b_{3}^{4}c_{1}^{2}c_{2}c_{3}-8\,a_{2}^{4}b_{2}^{2}b_{3}^{2}c_{1}c_{2}^{2}c_{3}\\
     &+2\,a_{2}^{4}b_{1}b_{3}^{3}c_{1}c_{2}^{2}c_{3}-10\,a_{1}a_{2}^{3}b_{2}b_{3}^{3}c_{1}c_{2}^{2}c_{3}-2\,a_{1}^{2}a_{2}^{2}b_{3}^{4}c_{1}c_{2}^{2}c_{3}+2\,a_{2}^{4}b_{2}b_{3}^{3}c_{2}^{3}c_{3}+2\,a_{1}a_{2}^{3}b_{3}^{4}c_{2}^{3}c_{3}+a_{2}^{4}b_{2}^{2}b_{3}^{2}c_{1}^{2}c_{3}^{2}\\
     &-12\,a_{2}^{4}b_{1}b_{3}^{3}c_{1}^{2}c_{3}^{2}-8\,a_{1}a_{2}^{3}b_{2}b_{3}^{3}c_{1}^{2}c_{3}^{2}-8\,a_{1}^{2}a_{2}^{2}b_{3}^{4}c_{1}^{2}c_{3}^{2}-10\,a_{2}^{4}b_{2}b_{3}^{3}c_{1}c_{2}c_{3}^{2}-8\,a_{1}a_{2}^{3}b_{3}^{4}c_{1}c_{2}c_{3}^{2}+a_{2}^{4}b_{3}^{4}c_{2}^{2}c_{3}^{2}-4\,a_{2}^{4}b_{3}^{4}c_{1}c_{3}^{3}; \\ ~ \\
    B & = -4\,a_{2}^{3}b_{2}^{3}c_{1}^{3}+18\,a_{2}^{3}b_{1}b_{2}b_{3}c_{1}^{3}-6\,a_{1}a_{2}^{2}b_{2}^{2}b_{3}c_{1}^{3}+36\,a_{1}a_{2}^{2}b_{1}b_{3}^{2}c_{1}^{3}+6\,a_{1}^{2}a_{2}b_{2}b_{3}^{2}c_{1}^{3}+4\,a_{1}^{3}b_{3}^{3}c_{1}^{3}+6\,a_{2}^{3}b_{2}^{2}b_{3}c_{1}^{2}c_{2}\\
     &+18\,a_{2}^{3}b_{1}b_{3}^{2}c_{1}^{2}c_{2}+24\,a_{1}a_{2}^{2}b_{2}b_{3}^{2}c_{1}^{2}c_{2}+6\,a_{1}^{2}a_{2}b_{3}^{3}c_{1}^{2}c_{2}+6\,a_{2}^{3}b_{2}b_{3}^{2}c_{1}c_{2}^{2}-6\,a_{1}a_{2}^{2}b_{3}^{3}c_{1}c_{2}^{2}-4\,a_{2}^{3}b_{3}^{3}c_{2}^{3}\\
     &+18\,a_{2}^{3}b_{2}b_{3}^{2}c_{1}^{2}c_{3}+36\,a_{1}a_{2}^{2}b_{3}^{3}c_{1}^{2}c_{3}+18\,a_{2}^{3}b_{3}^{3}c_{1}c_{2}c_{3}; \\ ~ \\
     C & = -27a_2^2b_3^2c_1^2. 
 \end{align*}
 \end{footnotesize}}
Note that 
 \begin{align*}
     \sqrt{(A,B,C)}^{\rm sat} & = (b_3c_1, a_2b_3c_2, a_2b_2c_1)  \\
     & = (a_2,b_3) \cap (a_2,c_1) \cap (b_2,b_3) \cap (b_3,c_1) \cap (c_1,c_2),
 \end{align*}
 where the saturation is intended with respect to the irrelevant ideal of the multigraded ring $\bbC[\bfa]\otimes\bbC[\bfb]\otimes\bbC[\bfc]$. Hence, if we restrict \eqref{eq:orbit16_firstinclusion} with this condition, we deduce 
 \begin{align}
    \calF_{T_{16}} \subset ~ & [Z(b_3c_1,a_2b_3c_2,a_2b_2c_1) \smallsetminus {(Z(a_2,b_2,b_3)\cup Z(a_2,c_1,c_3))}]
    \nonumber \\ & \cup [Z(b_3c_1,a_2b_3c_2,a_2b_2c_1) 
 \cap V(T_{16}^*)].
 \end{align}

 {\it Step 3.} Now, we apply \Cref{lemma:rank4_characterization}. If $T_{16} - \lambda P$ belongs to the orbit n.\protect\idref{14} for some $\lambda$, then the pencil of $3 \times 3$ matrices given by $T_{16} - \lambda P$ contains a rank-one matrix. The radical ideal of $2\times 2$ minors of $T_{16}-\lambda P$, saturated with respect to $(u,v)$, is 
  \[
 (c_1,b_3,u,\lambda(a_2b_1c_2 +a_2b_2c_3)-1).
    \]
 Hence, by \Cref{lemma:rank4_characterization}, under the assumptions $c_1 = b_3 = 0$ and $a_2(b_1c_2+b_2c_3) \neq 0$, we need to understand the intersection of the pencil with the secant variety of $X_{33}$. In particular, under these conditions, the determinant of the pencil $T-\lambda P$ is 
 \[
 \left(-\lambda\,a_{1}b_{2}c_{2}+1\right)u^{3}-\lambda\,a_{2}b_{2}c_{2}u^{2}v
 \]
 which is identically a pure cube if $a_2b_2c_2 =0$.

 Note that $(b_3c_1,a_2b_3c_2,a_2b_2c_1) = \left(b_{3}c_{1},\,a_{2}b_{3}c_{2},\,a_{2}b_{2}c_{2},\,a_{2}b_{2}c_{1}\right) \cap (b_3,c_1)$.

{Therefore, we conclude the following.
\begin{corollary}  
    The forbidden locus of the tensor $T_{16} \in \bbC^2 \otimes \bbC^3 \otimes \bbC^3$ defined in \Cref{table:kronecker} is
    \begin{align}
        \calF_{T_{16}}  =~ & \Big[[Z(b_3c_1,a_2b_3c_2,a_2b_2c_1) \smallsetminus {(Z(a_2,b_2,b_3)\cup Z(a_2,c_1,c_3))} ] 
        \cup [Z(b_3c_1,a_2b_3c_2,a_2b_2c_1) 
        \cap V(T_{16}^*)]\Big] \nonumber \\
        & \smallsetminus [Z(b_3,c_1) \smallsetminus Z(a_2b_1c_2+a_2b_2c_3, a_2b_2c_2)].
    \end{align}
\end{corollary}
}

\subsubsection{Orbit n.\protect\idref{17}}\label{ssec:n17}
Consider the generator of the orbit 
\[
T_{17} = \bfe_1^1\otimes (\bfe_1^2\otimes \bfe_1^3+\bfe_2^2\otimes \bfe_2^3)+\bfe_2^1\otimes (\bfe_1^2\otimes \bfe_2^3+\bfe_3^2\otimes \bfe_3^3).
\]
Let $P = \bfa \otimes \bfb \otimes \bfc$.

\textit{Step 1.} The ideal of maximal minors of all flattenings of $T_{17}-\lambda P$ is 
\[
 \Big[(a_1,b_1,b_2) \cap (a_2,b_2,b_3) \cap (a_1,c_1,c_2) \cap (a_2,c_1,c_3)\Big]
 + \left(\lambda\langle T_{17}^*,P\rangle - 1\right).
\]
Hence, 
\begin{align}\label{eq:orbit17_firstinclusion}
    \calF_{T_{17}} \subset \Big[\bbC^2\otimes\bbC^3\otimes\bbC^3 \smallsetminus \left( Z(a_1,b_1,b_2) \cup Z(a_2,b_2,b_3) \cup Z(a_1,c_1,c_2) \cup Z(a_2,c_1,c_3)\right)\Big] \cup V(T_{17}^*).
\end{align}

\textit{Step 2.} We evaluate the hyperdeterminant at $T_{17}-\lambda P$: 
\[
    H_{233}(T_{17}-\lambda P)  = \lambda (A\lambda^3 + B\lambda^2 + C\lambda + D)
 \]
 where
 {\footnotesize
 \begin{align*}
     A & = a_{1}^{2}a_{2}^{2}b_{1}^{4}c_{1}^{4}+2\,a_{1}^{3}a_{2}b_{1}^{3}b_{2}c_{1}^{4}+a_{1}^{4}b_{1}^{2}b_{2}^{2}c_{1}^{4}+4\,a_{1}^{2}a_{2}^{2}b_{1}^{3}b_{2}c_{1}^{3}c_{2}+ 6\,a_{1}^{3}a_{2}b_{1}^{2}b_{2}^{2}c_{1}^{3}c_{2}+2\,a_{1}^{4}b_{1}b_{2}^{3}c_{1}^{3}c_{2}\\
     & +6\,a_{1}^{2}a_{2}^{2}b_{1}^{2}b_{2}^{2}c_{1}^{2}c_{2}^{2}+6\,a_{1}^{3}a_{2}b_{1}b_{2}^{3}c_{1}^{2}c_{2}^{2}+a_{1}^{4}b_{2}^{4}c_{1}^{2}c_{2}^{2}+4\,a_{1}^{2}a_{2}^{2}b_{1}b_{2}^{3}c_{1}c_{2}^{3}+2\,a_{1}^{3}a_{2}b_{2}^{4}c_{1}c_{2}^{3}+a_{1}^{2}a_{2}^{2}b_{2}^{4}c_{2}^{4}\\
     &-2\,a_{1}a_{2}^{3}b_{1}^{3}b_{3}c_{1}^{3}c_{3}+2\,a_{1}^{2}a_{2}^{2}b_{1}^{2}b_{2}b_{3}c_{1}^{3}c_{3}+8\,a_{1}^{3}a_{2}b_{1}b_{2}^{2}b_{3}c_{1}^{3}c_{3}+4\,a_{1}^{4}b_{2}^{3}b_{3}c_{1}^{3}c_{3}-6\,a_{1}a_{2}^{3}b_{1}^{2}b_{2}b_{3}c_{1}^{2}c_{2}c_{3}\\
     &+4\,a_{1}^{2}a_{2}^{2}b_{1}b_{2}^{2}b_{3}c_{1}^{2}c_{2}c_{3}+8\,a_{1}^{3}a_{2}b_{2}^{3}b_{3}c_{1}^{2}c_{2}c_{3}-6\,a_{1}a_{2}^{3}b_{1}b_{2}^{2}b_{3}c_{1}c_{2}^{2}c_{3}+2\,a_{1}^{2}a_{2}^{2}b_{2}^{3}b_{3}c_{1}c_{2}^{2}c_{3}-2\,a_{1}a_{2}^{3}b_{2}^{3}b_{3}c_{2}^{3}c_{3}\\
     &+a_{2}^{4}b_{1}^{2}b_{3}^{2}c_{1}^{2}c_{3}^{2}-8\,a_{1}a_{2}^{3}b_{1}b_{2}b_{3}^{2}c_{1}^{2}c_{3}^{2}-8\,a_{1}^{2}a_{2}^{2}b_{2}^{2}b_{3}^{2}c_{1}^{2}c_{3}^{2}+2\,a_{2}^{4}b_{1}b_{2}b_{3}^{2}c_{1}c_{2}c_{3}^{2}-8\,a_{1}a_{2}^{3}b_{2}^{2}b_{3}^{2}c_{1}c_{2}c_{3}^{2}\\
     &+a_{2}^{4}b_{2}^{2}b_{3}^{2}c_{2}^{2}c_{3}^{2}+4\,a_{2}^{4}b_{2}b_{3}^{3}c_{1}c_{3}^{3}; \\
     ~\\
     B & = -2\,a_{1}a_{2}^{2}b_{1}^{3}c_{1}^{3}-8\,a_{1}^{2}a_{2}b_{1}^{2}b_{2}c_{1}^{3}-2\,a_{1}^{3}b_{1}b_{2}^{2}c_{1}^{3}-6\,a_{1}a_{2}^{2}b_{1}^{2}b_{2}c_{1}^{2}c_{2}-16\,a_{1}^{2}a_{2}b_{1}b_{2}^{2}c_{1}^{2}c_{2}-2\,a_{1}^{3}b_{2}^{3}c_{1}^{2}c_{2}\\
     & -6\,a_{1}a_{2}^{2}b_{1}b_{2}^{2}c_{1}c_{2}^{2}-8\,a_{1}^{2}a_{2}b_{2}^{3}c_{1}c_{2}^{2}-2\,a_{1}a_{2}^{2}b_{2}^{3}c_{2}^{3}-2\,a_{2}^{3}b_{1}^{2}b_{3}c_{1}^{2}c_{3}-2\,a_{1}a_{2}^{2}b_{1}b_{2}b_{3}c_{1}^{2}c_{3}-20\,a_{1}^{2}a_{2}b_{2}^{2}b_{3}c_{1}^{2}c_{3}\\
     & -4\,a_{2}^{3}b_{1}b_{2}b_{3}c_{1}c_{2}c_{3}-2\,a_{1}a_{2}^{2}b_{2}^{2}b_{3}c_{1}c_{2}c_{3}-2\,a_{2}^{3}b_{2}^{2}b_{3}c_{2}^{2}c_{3}-12\,a_{2}^{3}b_{2}b_{3}^{2}c_{1}c_{3}^{2}; \\
     ~\\
     C & = a_{2}^{2}b_{1}^{2}c_{1}^{2}+10\,a_{1}a_{2}b_{1}b_{2}c_{1}^{2}+a_{1}^{2}b_{2}^{2}c_{1}^{2}+2\,a_{2}^{2}b_{1}b_{2}c_{1}c_{2}+10\,a_{1}a_{2}b_{2}^{2}c_{1}c_{2}+a_{2}^{2}b_{2}^{2}c_{2}^{2}+12\,a_{2}^{2}b_{2}b_{3}c_{1}c_{3}; \\
     ~\\
     D & = -4\,a_{2}b_{2}c_{1}.
 \end{align*}
 }
 Note that
 \begin{align*}
    \sqrt{(A,B,C,D)}^{\rm sat} & = \left(b_{2}c_{1},\,a_{2}b_{2}c_{2},\,a_{2}b_{1}c_{1}\right) \\
     & (a_2,b_2) \cap (a_2,c_1) \cap (b_1,b_2) \cap (b_2,c_1) \cap (c_1,c_2),
 \end{align*} 
 where the saturation is intended with respect to the irrelevant ideal of the multigraded ring $\bbC[\bfa]\otimes\bbC[\bfb]\otimes\bbC[\bfc]$.
 Hence, if we restrict \eqref{eq:orbit17_firstinclusion} with this conditions, we deduce
 \begin{align}\label{eq:orbit17_secondinclusion}
    \calF_{T_{17}} \subset ~ & [Z(b_2c_1,a_2b_2c_2,a_2b_1c_1) \smallsetminus (Z(a_1,b_1,b_2) \cup Z(a_2,b_2,b_3) \cup Z(a_1,c_1,c_2) \cup Z(a_2,c_1,c_3))] \nonumber \\ & \cup [Z(b_2c_1,a_2b_2c_2,a_2b_1c_1) 
 \cap V(T_{17}^*)].
 \end{align}
 {\it Step 3.} We apply \Cref{lemma:rank4_characterization}. First we check the intersections of the pencil $T_{17} - \lambda P$ with the Segre $X_{33}$: the radical of the ideal of $2 \times 2$ minors, saturated with respect to  $(u,v)$, is  
 \[ 
  \left(c_{1},\,b_{2},\,u,\,\lambda(a_{2}b_{1}c_{2}+\,a_{2}b_{3}c_{3})-1\right) \cap \left(c_{3},\,c_{1},\,b_{3},\,a_{2},\,v,\,\lambda\,a_{1}b_{2}c_{2}-1\right) \cap \left(c_{3},\,b_{3},\,b_{2},\,a_{2},\,v,\,\lambda\,a_{1}b_{1}c_{1}-1\right).
 \]
 Hence, we have three numerical conditions to check also the intersection of the pencil with $\sigma_2(X_{3,3})$. 

 The determinant of the pencil $T_{17}-\lambda P$ becomes: 
 \begin{itemize}
     \item if $b_2 = c_1 = 0$: 
     \[
      \det(T_{17}-\lambda P) = -\lambda\,a_{1}b_{3}c_{3}u^{3}+\left(-\lambda\,a_{2}b_{3}c_{3}+1\right)u^{2}v
     \]
     which is never identically zero or a pure cube;
     \item if $a_2 = b_3 = c_1 = c_3 = 0$: 
     \[
     \det(T_{17}-\lambda P) = \left(-\lambda\,a_{1}b_{2}c_{2}+1\right)u^{2}v
     \]
     which is never identically zero or a pure cube;
     \item if $a_2 = b_2 = b_3 = c_3 = 0$:
     \[
     \det (T_{17}-\lambda P) = \left(-\lambda\,a_{1}b_{1}c_{1}+1\right)u^{2}v
         \]
     which is never identically zero or a pure cube.
 \end{itemize}
 In conclusion, we have that $T-\lambda P$ never intersects $\calO_{14}$, namely it is all contained in $\calO_{17}$. In particular,  \eqref{eq:orbit17_secondinclusion} is an equality.
{
 \begin{corollary}  
     The forbidden locus of the tensor $T_{17} \in \bbC^2 \otimes \bbC^3 \otimes \bbC^3$ defined in \Cref{table:kronecker} is
     \begin{align}
        \calF_{T_{17}} = \Big[\bbC^2\otimes\bbC^3\otimes\bbC^3 \smallsetminus \left( Z(a_1,b_1,b_2) \cup Z(a_2,b_2,b_3) \cup Z(a_1,c_1,c_2) \cup Z(a_2,c_1,c_3)\right)\Big] \cup V(T_{17}^*).
    \end{align}
 \end{corollary}
 }

 \subsection{Orbits n.\protect\idref{19}, n.\protect\idref{20}, n.\protect\idref{22} and n.\protect\idref{23}: rank-$4$ tensors in $\bbC^2\otimes\bbC^3\otimes\bbC^4$}
In $\mathbb{C}^2\otimes \mathbb{C}^3\otimes \mathbb{C}^4$ concise tensors have rank at least 4 and the maximal rank is 5. Following \cref{table:kronecker}, the normal forms of rank-4 tensors in $\mathbb{C}^2\otimes \mathbb{C}^3\otimes \mathbb{C}^4$ are:
\begin{equation}\label{eq:orbits_234}
\begin{aligned}
    T_{19}=&~\bf{e}_1^1\otimes (\bf{e}_1^2\otimes \bf{e}_1^3+\bf{e}_2^2\otimes \bf{e}_2^3+\bf{e}_3^2\otimes \bf{e}_4^3) + \bf{e}_2^1 \otimes (\bf{e}_1^2\otimes \bf{e}_2^3+\bf{e}_2^2\otimes \bf{e}^3_3),\\
     T_{20}=&~\bf{e}^1_1\otimes (\bf{e}^2_1\otimes \bf{e}^3_1+\bf{e}^2_2\otimes \bf{e}^3_3+\bf{e}^2_3\otimes \bf{e}^3_4) + \bf{e}^1_2 \otimes (\bf{e}^2_1\otimes \bf{e}^3_2),\\
      T_{22}=&~\bf{e}^1_1\otimes (\bf{e}^2_1\otimes \bf{e}^3_1+\bf{e}^2_2\otimes \bf{e}^3_3) + \bf{e}^1_2 \otimes (\bf{e}^2_1\otimes \bf{e}^3_2+\bf{e}^2_3\otimes \bf{e}^3_4),\\
     T_{23}=&~\bf{e}^1_1\otimes (\bf{e}^2_1\otimes \bf{e}^3_1+\bf{e}^2_2\otimes \bf{e}^3_2+\bf{e}^2_3\otimes \bf{e}^3_3) + \bf{e}^1_2 \otimes (\bf{e}^2_1\otimes \bf{e}^3_2+\bf{e}^2_2\otimes \bf{e}^3_3+ \bf{e}^2_3\otimes \bf{e}^3_4).
     \end{aligned}
\end{equation}
Recall that in $\mathbb{C}^2 \otimes \mathbb{C}^3 \otimes \mathbb{C}^4$ there is only one canonical form of rank 5 which is $T_{21}$ in \cref{table:kronecker}. 

We give a general description of the decomposition loci of tensors in the orbits  \eqref{eq:orbits_234}. 

\begin{lemma}\label{lemma:rank4_234}
    Let $T$ be a tensor in the orbits \eqref{eq:orbits_234}. Then, $P = \bfa \otimes \bfb \otimes \bfc \in \calD_T$ if and only if one of the following conditions occurs:
    \begin{enumerate}[label=(\roman*), ref=\thelemma-(\roman*)]
        \item\label{lemma:rank4_234_cond1}  either there exists $\lambda \in \bbC$ such that $T - \lambda P \in \Sub_{134}$; 
        \item\label{lemma:rank4_234_cond2} or there exists $\lambda \in \bbC$ such that  
        \begin{enumerate}
            \item either $T-\lambda P \in \Sub_{223}$; 
            \item or $T - \lambda P \in (\Sub_{233} \smallsetminus \calX^\vee_{233}) \cup (\Sub_{233} \cap \calO_{14})$.
        \end{enumerate}
    \end{enumerate}
\end{lemma}
\begin{proof}
    It follows directly from \cref{table:kronecker} (cf. also \cite[Tables 1, 3 and 4]{BL}). Indeed:
     \begin{itemize}
         \item if $T-\lambda P \in \bbC^1 \otimes \bbC^3 \otimes \bbC^4$, then we reduce the matrix case and $\rk(T-\lambda P) \leq 3$;
         \item if $T-\lambda P \in \bbC^2 \otimes \bbC^2 \otimes \bbC^3$ then its rank is at most $3$, see \Cref{sec:22n};
         \item if $T-\lambda P \in \bbC^2 \otimes \bbC^3 \otimes \bbC^3$ then it has rank at most three unless it belongs to the dual Segre variety $\calX_{233}^\vee$ but not to the orbit $\calO_{14}$, see \Cref{sec:23n}.
     \end{itemize}
     In all other cases $T-\lambda P$ is concise and has rank at least $4$.
\end{proof}

\begin{algorithm}[htb]
\caption{Test if a rank-$1$ tensor lies in the decomposition locus of a rank-$4$ tensor in $\bbC^2 \otimes \bbC^3 \otimes \bbC^4$}
\flushleft{
\textbf{Input:} }
\begin{itemize}
    \item Rank-$4$ concise tensor $T \in \bbC^2 \otimes \bbC^3 \otimes \bbC^4$.
    \item Rank-one tensor $P \in \bbC^2 \otimes \bbC^3 \otimes \bbC^4$.
\end{itemize}
\flushleft{
\textbf{Output:} Whether $P \in \calD_T$ or $P \not\in \calD_T$.
}

\flushleft{
\textbf{Procedure:}
}
\begin{enumerate}
    \item[{\it Step 1.}] Compute the principal ideal $I_1 := (f_1) \in \bbC[\lambda]$ of maximal minors of the first flattening of $T-\lambda P$.
    \begin{itemize}
        \item If $I_1 \neq (1)$ then $T-\lambda P \in \Sub_{134}$ for all roots of $f_1$, \\ hence \textbf{return} $P \in \calD_T$;
        \item otherwise, we proceed to \textit{Step 2}.
    \end{itemize}
    \item[{\it Step 2.}] Compute the principal ideals $I_2 := (f_2), I_3 := (f_3) \in \bbC[\lambda]$ of maximal minors of the second and third flattening of $T-\lambda P$, respectively.
    \begin{itemize}
        \item If $I_3 \neq (1)$, then $\rk(T-\lambda P)$ is at least four for all $\lambda$'s; \\
        then, \textbf{return} $P \not\in \calD_T$;
        \item else, if $f_2$ and $f_3$ have a common root $\lambda$, then $T-\lambda P \in \Sub_{223}$ and it has rank at most $3$; \\
        then, \textbf{return} $P \in \calD_T$;
        \item else, if there exists a root $\lambda$ of $f_3$ such that $H_{233}(T-\lambda P) \neq 0$, then $T-\lambda P \in \Sub_{233} \smallsetminus \calX_{233}^\vee$ and it has rank $3$; \\
        then, \textbf{return} $P \in \calD_T$;
        \item else, if there exists a root $\lambda$ of $f_3$ such that $H_{233}(T-\lambda P) = 0$ but $T-\lambda P \in \calO_{14}$ (which can be checked as explained in \Cref{lemma:rank4_characterization} and \Cref{algo:233}), then it has rank $3$; \\
        then, \textbf{return} $P \in \calD_T$;
        \item else, if none of the above is satisfied,\\ 
        then \textbf{return} $P \not\in \calD_T$.
    \end{itemize}
\end{enumerate}
\label{algo:234}
\end{algorithm}

We use \Cref{algo:234} on each one of the orbits in \eqref{eq:orbits_234}. Similarly to \Cref{ssec:rank4_233}, we use the structure of \Cref{algo:234} for symbolic computations by working on $\bbC[\bfa,\bfb,\bfc]$ and assuming $P = \bfa\otimes\bfb\otimes\bfc$. Note that the application of \Cref{lemma:rank4_characterization} in \textit{Step 2} of \Cref{algo:234} involves a higher level of intricacy compared to its use in \Cref{ssec:rank4_233}. Indeed, once we know the constraints on $P$ and $\lambda$ such that $T-\lambda P \in \Sub_{233}$, then we need to perform the change of coordinates in order to explicitly express $T-\lambda P$ as a tensor in $\bbC^2\otimes\bbC^3\otimes\bbC^3$ and then make the  computations as in \Cref{ssec:rank4_233}. Such a change of coordinates cannot be uniquely chosen, but need to be changed depending on conditions on $\bfa,\bfb,\bfc$.

\subsubsection{Orbit n.\protect\idref{19}}\label{ssec:n19}
Let $T_{19} = \bfe_1^1\otimes (\bfe_1^2\otimes \bfe_1^3+\bfe_2^2 \otimes \bfe_2^3+\bfe_3^2\otimes \bfe_4^3)+\bfe_2^1\otimes (\bfe_1^2\otimes \bfe_2^3+\bfe_2^2\otimes \bfe_3^3)$ 
and $P = \bfa\otimes\bfb\otimes\bfc$.

\textit{Step 1.} The ideal of maximal minors of the first flattening of $T_{19} - \lambda P$ is $I_1 = (1)$. 

\textit{Step 2.} The ideal of maximal minors of the second and third flattening of $T_{19}-\lambda P$ are, respectively:
\begin{align*}
    I_2 & = (c_3,c_2,c_1,a_2,a_1b_3c_4\lambda-1), \\
    I_3 & =(b_2b_3,a_2b_3,a_2b_1-a_1b_2,\lambda(a_1b_1c_1+a_1b_2c_2+a_2b_2c_3+a_1b_3c_4)-1).
\end{align*}
Hence, we have that \[\calD_T \subset Z(b_2b_3,a_2b_3,a_2b_1-a_1b_2) \smallsetminus Z(a_1b_1c_1+a_1b_2c_2+a_2b_2c_3+a_1b_3c_4).\]

In the ring $\bbC[\lambda,\bfa,\bfb,\bfc]/(b_2b_3,a_2b_3,a_2b_1-a_1b_2,\lambda(a_1b_1c_1+a_1b_2c_2+a_2b_2c_3+a_1b_3c_4)-1)$, we get that $I_2 = (a_2,c_1,c_2,c_3)$: hence, by \Cref{lemma:rank4_234}, under these conditions $T_{19}-\lambda P \in \Sub_{223}$ and $P \in \calD_{T_{19}}$.

We need to check under which conditions $T_{19}-\lambda P \in \calX_{233}^\vee$. In the ring $\bbC[\lambda,\bfa,\bfb,\bfc]/(b_2b_3,a_2b_3,a_2b_1-a_1b_2,\lambda(a_1b_1c_1+a_1b_2c_2+a_2b_2c_3+a_1b_3c_4)-1)$, the ideal of maximal minors of the pencil $T_{19}-\lambda P$ of $3 \times 4$ matrices is equal to $(c_3u^3 - c_2u^2v + c_1uv^2)$. 
The discriminant of the principal generator is $c_1^2(4c_1c_3-c_2^2)$. Hence, if $c_1(4c_1c_3-c_2^2) \neq 0$ then $T_{19}-\lambda P \not\in \calX_{233}^\vee$ and $P \in \calD_{T_{19}}$.

Assuming $c_1(4c_1c_3-c_2^2) = 0$, we check if $T_{19}-\lambda P \in \calO_{14}$ by employing \Cref{lemma:rank4_characterization}. The saturation with respect $(u,v) $ of the ideal of $2 \times 2$ minors of the pencil $T_{19}-\lambda P$ is equal to $(b_3,c_1,c_4,u)$. 
Imposing $b_3 = c_1 = c_4 = 0$, the radical ideal of $3 \times 3$ minors becomes $(c_3u^2-c_2uv)$ 
which has two distinct roots unless $c_2 = 0$. 
{Therefore, we deduce the following.
 \begin{corollary}  
     The forbidden locus of the tensor $T_{19} \in \bbC^2 \otimes \bbC^3 \otimes \bbC^4$ defined in \Cref{table:kronecker} is
     \begin{align*}
        \calD_{T_{19}} = & ~\Big[[Z(b_2b_3,a_2b_3,a_2b_1-a_1b_2) \smallsetminus Z(a_1b_1c_1+a_1b_2c_2+a_2b_2c_3+a_1b_3c_4)] \\
        & \qquad\qquad\qquad\qquad\qquad\qquad \smallsetminus \big[Z(c_1(4c_1c_3-c_2^2)) \smallsetminus [Z(b_3,c_1,c_4) \smallsetminus Z(c_2)]\big]\Big] \\
        & \cup \Big[[Z(b_2b_3,a_2b_3,a_2b_1-a_1b_2) \smallsetminus Z(a_1b_1c_1+a_1b_2c_2+a_2b_2c_3+a_1b_3c_4)] \cap Z(a_2,c_1,c_2,c_3)\Big].
    \end{align*}
 \end{corollary}
 }

\subsubsection{Orbit n.\protect\idref{20}}\label{ssec:n20}
Let $T_{20} = \bfe_{1}^1\otimes (\bfe_{1}^2\otimes \bfe_{1}^3+\bfe_{2}^2\otimes \bfe_{3}^3+\bfe_{3}^2\otimes \bfe_{4}^3)+\bfe_{2}^1\otimes \bfe_{1}^2\otimes \bfe_{2}^3$ and $P = \bfa \otimes \bfb \otimes \bfc$. 

\textit{Step 1.} The radical ideal of maximal minors of the first flattening of $T_{20}-\lambda P$ is
\[
    I_1 = \left(c_{4},\,c_{3},\,c_{1},\,b_{3},\,b_{2},\,\lambda a_{2}b_{1}c_{2}-1\right).
\]
Hence, $Z(b_3,{b_2}, c_1,c_3,c_4) \smallsetminus Z(a_1b_1c_2) \subset \calD_{T_{20}}$.

\textit{Step 2.} The radical ideals of maximal minors of second and third flattenings of $T_{20}-\lambda P$ are
\begin{align*}
    I_2 & = \left(c_{2},\,c_{1},\,a_{2},\lambda(a_{1}b_{2}c_{3}+ a_{1}b_{3}c_{4})-1\right),\\
    I_3 & = \left(a_{2}b_{3},\,a_{2}b_{2},\lambda(a_{1}b_{1}c_{1}+a_{2}b_{1}c_{2}+a_{1}b_{2}c_{3}+a_{1}b_{3}c_{4})-1\right),
\end{align*}
respectively.

In the ring $\bbC[\lambda,\bfa,\bfb,\bfc]/(a_{2}b_{3},\,a_{2}b_{2},\lambda(a_{1}b_{1}c_{1}+a_{2}b_{1}c_{2}+a_{1}b_{2}c_{3}+a_{1}b_{3}c_{4})-1)$, we ge that $I_2 = (a_2,c_1,c_2)$: hence, by \Cref{lemma:rank4_234}, under these conditions, $T_{20}-\lambda P\in \Sub_{223}$ and $P \in \calD_{T_{20}}$.

We need to check under which conditions $T_{20}-\lambda P \in \calX_{233}^\vee$. In the ring $\bbC[\lambda,\bfa,\bfb,\bfc]/(a_{2}b_{3},\,a_{2}b_{2},\lambda(a_{1}b_{1}c_{1}+a_{2}b_{1}c_{2}+a_{1}b_{2}c_{3}+a_{1}b_{3}c_{4})-1)$, the ideal of maximal minors of the pencil $T_{20}-\lambda P$ is equal to $(c_2u^3-c_1u^2v)$.   
The discriminant of the generator is identically zero, i.e., $T_{20}-\lambda P \in \calX_{233}^\vee$. We check if it belongs to $\calO_{14}$ by employing \Cref{lemma:rank4_characterization}. The saturation with respect to $(u,v)$ 
of the ideal of $2 \times 2$ minors of the pencil $T_{20}-\lambda P$ is $(u)$:  
hence, the pencil always contains a rank-one matrix. The radical ideal of maximal minors is generated by $c_2u^2-c_1uv$ 
which has two distinct roots unless $c_1 = 0$. {Therefore, we deduce the following.
 \begin{corollary}  
     The forbidden locus of the tensor $T_{20} \in \bbC^2 \otimes \bbC^3 \otimes \bbC^4$ defined in \Cref{table:kronecker} is
     \begin{align*}
    \calD_{T_{20}} =~ & \Big[Z(b_3,{b_2}, c_1,c_3,c_4)\smallsetminus Z(a_1b_1c_2)\Big] \\ & \cup \Big[[Z(a_{2}b_{3},\,a_{2}b_{2}) \smallsetminus V(T^*_{20})
    ] \smallsetminus Z(c_1) \Big] \\
    & \cup \Big[[Z(a_{2}b_{3},\,a_{2}b_{2}) \smallsetminus V(T_{20}^*) 
    ] \cap Z(a_2,c_1,c_2)\Big].
\end{align*}
 \end{corollary}
 }

\subsubsection{Orbit n.\protect\idref{22}}\label{ssec:n22}
Let $T_{22} = \bfe_1^1\otimes (\bfe_1^2\otimes \bfe_1^3+\bfe_2^2\otimes \bfe_3^3)+\bfe_2^1\otimes (\bfe_1^2\otimes \bfe_2^3+\bfe_3^2\otimes \bfe_4^3)$ and $P = \bfa \otimes \bfb \otimes \bfc$. 

\textit{Step 1.} The radical ideal of maximal minors of the first flattening of $T_{22}-\lambda P$ is $I_1 = (1)$.

\textit{Step 2.} The radical ideals of maximal minors of second and third flattenings of $T_{22}-\lambda P$ are 
\begin{align*}
    I_2 & = \left(c_{2},\,c_{1},\,c_{3}c_{4},\,a_{1}c_{4},\,a_{2}c_{3},\,a_{1}a_{2},\,\lambda(a_{1}b_{2}c_{3}+a_{2}b_{3}c_{4})-1\right), \\
    I_3 & = \left(b_{2}b_{3},\,a_{1}b_{3},\,a_{2}b_{2},\lambda(a_{1}b_{1}c_{1}+a_{2}b_{1}c_{2}+a_{1}b_{2}c_{3}+a_{2}b_{3}c_{4})-1\right)
\end{align*}
respectively.

In the ring $\bbC[\lambda,\bfa,\bfb,\bfc]/(b_{2}b_{3},\,a_{1}b_{3},\,a_{2}b_{2},\lambda(a_{1}b_{1}c_{1}+a_{2}b_{1}c_{2}+a_{1}b_{2}c_{3}+a_{2}b_{3}c_{4})-1)$, we get \[I_2 = \left(c_{2},\,c_{1},\,c_{3}c_{4},\,a_{1}c_{4},\,a_{2}c_{3},\,a_{1}a_{2}\right).\] By \Cref{lemma:rank4_234}, under these conditions $T_{22}-\lambda P \in \Sub_{223}$ and $P \in \calD_{T_{22}}$.

We check under which conditions $T_{22}-\lambda P \in \calX_{233}^\vee$. In the ring $\bbC[\lambda,\bfa,\bfb,\bfc]/(b_{2}b_{3},\,a_{1}b_{3},\,a_{2}b_{2},\lambda(a_{1}b_{1}c_{1}+a_{2}b_{1}c_{2}+a_{1}b_{2}c_{3}+a_{2}b_{3}c_{4})-1)$, the radical ideal of maximal minors of the pencil $T_{22}-\lambda P$ is equal to $(c_2u^2v-c_1uv^2)$ whose discriminant of the generator is $c_1^2c_2^2$. Hence, if $c_1c_2 \neq 0$, then $T_{22}-\lambda P \not\in \calX_{233}^\vee$ and $P \in \calD_{T_{22}}$. We check if $T_{22}-\lambda P \in \calO_{14}$ by employing \Cref{lemma:rank4_characterization}. The radical ideal of $3 \times 3$ minors of the pencil $T_{22}-\lambda P$ inside the ring $\bbC[\lambda,\bfa,\bfb,\bfc]/(b_{2}b_{3},\,a_{1}b_{3},\,a_{2}b_{2},c_1c_2,\lambda(a_{1}b_{1}c_{1}+a_{2}b_{1}c_{2}+a_{1}b_{2}c_{3}+a_{2}b_{3}c_{4})-1)$ is equal to the ideal $(c_2uv,c_1uv,a_1a_2uv) = (u)\cap(v)\cap(c_2,c_1,a_1)\cap(c_2,c_1,a_2)$.  
In particular, we observe that the pencil $T_{22}-\lambda P$ it is either fully contained in the secant variety of $X_{33}$ or it intersects it in only one point: by \Cref{lemma:rank4_characterization}, that means that $T_{22}-\lambda P$ never belongs to $\calO_{14}$.

{Therefore, we deduce the following.
 \begin{corollary}  
     The forbidden locus of the tensor $T_{22} \in \bbC^2 \otimes \bbC^3 \otimes \bbC^4$ defined in \Cref{table:kronecker} is
     \begin{align*}
        \calD_{T_{22}} =~  & \Big[[Z(b_{2}b_{3},\,a_{1}b_{3},a_2b_2) \smallsetminus V(T^*_{22})
        ] \smallsetminus Z(c_1c_2) \Big] \\
        & \cup \Big[[Z(b_{2}b_{3},\,a_{1}b_{3},a_2b_2) \smallsetminus
        V(T^*_{22})
        ] \cap Z(c_{2},\,c_{1},\,c_{3}c_{4},\,a_{1}c_{4},\,a_{2}c_{3},\,a_{1}a_{2})\Big].
    \end{align*}
 \end{corollary}
 }

\subsubsection{Orbit n.\protect\idref{23}}\label{ssec:n23}
Let $T_{23} = \bfe_1^1\otimes (\bfe_1^2\otimes \bfe_1^3+\bfe_2^2\otimes \bfe_2^3+\bfe_3^2\otimes \bfe_3^3)+\bfe_2^1\otimes(\bfe_1^2\otimes \bfe_2^3+\bfe_2^2\otimes \bfe_3^3+\bfe_3^2 \otimes \bfe_4^3)$  
and $P = \bfa \otimes \bfb \otimes \bfc$. 

\textit{Step 1.} The radical ideal of maximal minors of the first flattening of $T_{23}-\lambda P$ is $I_1 = (1)$.

\textit{Step 2.} The radical ideals of maximal minors of second and third flattenings of $T_{22}-\lambda P$ are 
\begin{align*}
    I_2 & = (1), \\
    I_3 & = \left(b_{2}^{2}-b_{1}b_{3},\,a_{2}b_{2}-a_{1}b_{3},\,a_{2}b_{1}-a_{1}b_{2},\lambda(a_{1}b_{1}c_{1}+a_{1}b_{2}c_{2}+a_{1}b_{3}c_{3}+a_{2}b_{3}c_{4})-1\right)
\end{align*}
respectively.

We check under which conditions $T_{23}-\lambda P \in \calX_{233}^\vee$. In the ring $\bbC[\lambda,\bfa,\bfb,\bfc]/(b_{2}^{2}-b_{1}b_{3},\,a_{2}b_{2}-a_{1}b_{3},\,a_{2}b_{1}-a_{1}b_{2},\lambda(a_{1}b_{1}c_{1}+a_{1}b_{2}c_{2}+a_{1}b_{3}c_{3}+a_{2}b_{3}c_{4})-1)$, the ideal of maximal minors of the pencil $T_{23}-\lambda P$ is equal to $(c_{4}u^{3}-c_{3}u^{2}v+c_{2}uv^{2}-c_{1}v^{3})$   
whose discriminant of the generator is equal to $\Delta := -c_{2}^{2}c_{3}^{2}+4\,c_{1}c_{3}^{3}+4\,c_{2}^{3}c_{4}-18\,c_{1}c_{2}c_{3}c_{4}+27\,c_{1}^{2}c_{4}^{2}$. If $\Delta \neq 0$, then $T_{23}-\lambda P \not\in \calX_{233}^\vee$ and $P \in \calD_{T_{23}}$. We need to check if $T_{23}-\lambda P \in \calO_{14}$: the saturation of the ideal of $2 \times 2$ minors is equal to $(1)$, hence the pencil $T_{23}-\lambda P$ does not contain any rank-one matrix and, by \Cref{lemma:rank4_characterization}, it never belongs to $\calO_{14}$.
{Therefore, we deduce the following.
 \begin{corollary}  
     The forbidden locus of the tensor $T_{23} \in \bbC^2 \otimes \bbC^3 \otimes \bbC^4$ defined in \Cref{table:kronecker} is
     \begin{align*}
        \calD_{T_{23}} =~  & \Big[[Z(b_{2}^{2}-b_{1}b_{3},\,a_{2}b_{2}-a_{1}b_{3},\,a_{2}b_{1}-a_{1}b_{2}) \smallsetminus Z(a_{1}b_{1}c_{1}+a_{1}b_{2}c_{2}+a_{1}b_{3}c_{3}+a_{2}b_{3}c_{4})] \smallsetminus Z(\Delta) \Big],
     \end{align*}
    where $\Delta = -c_{2}^{2}c_{3}^{2}+4\,c_{1}c_{3}^{3}+4\,c_{2}^{3}c_{4}-18\,c_{1}c_{2}c_{3}c_{4}+27\,c_{1}^{2}c_{4}^{2}$.
 \end{corollary}
 }

 \subsection{Orbit n.\protect\idref{21}: rank-$5$ tensors in $\bbC^2\otimes\bbC^3\otimes\bbC^4$}\label{ssec:n21} There is one orbit of rank-$5$ tensors in $\bbC^2 \otimes \bbC^3 \otimes \bbC^4$:
 \[
    T_{21} = \bfe^1_1\otimes ( \bfe^2_1\otimes \bfe^3_1 +\bfe^2_2\otimes \bfe^3_3+ \bfe^2_3\otimes \bfe^3_4)+ \bfe^1_2\otimes ( \bfe^2_1\otimes \bfe^3_2 + \bfe^2_2\otimes \bfe^3_4).
 \]
 The orbits n.\idref{19}, n.\idref{20} and n.\idref{21} are the concise components of the singular locus of $\calX_{234}^\vee$ which is the Zariski closure of the orbit n.\idref{22}, see \cite[Table 4]{BL14}. The orbit n.\idref{23} consists of concise tensors which do not lie in $\calX_{234}^\vee$. Here, we compute the forbidden locus of $T_{21}$ by checking under which conditions $T-\lambda P \in \calO_{21}$ for all $\lambda$'s. In order to do that, we need to distinguish these orbits.
 \begin{lemma}\label{lemma:234_rank5}
     Let $T \in \calO_{19} \cup \calO_{20} \cup \calO_{21} \cup \calO_{22} \cup \calO_{23}$ and let $\ell_T \subset \bbP (\bbC^3 \otimes \bbC^4)$ be the corresponding pencil of matrices. Then:
     \begin{enumerate}
         \item if $T \in \calO_{19}$ then $\ell_T \cap \sigma_2(X_{34})$ is a single smooth point;
         \item if $T \in \calO_{20}$ then $\ell_T \cap \sigma_2(X_{34})$ is a $2$-jet supported at a rank-one point;
         \item if $T \in \calO_{21}$ then $\ell_T \cap \sigma_2(X_{34})$ is a $2$-jet supported at a rank-two point;
         \item if $T \in \calO_{22}$ then $\ell_T \cap \sigma_2(X_{34})$ is supported at two distinct points;
         \item if $T \in \calO_{23}$ then $\ell_T \cap \sigma_2(X_{34}) = \emptyset$.
     \end{enumerate}
 \end{lemma}
 \begin{proof}
    Since the rank of a matrix is invariant under the action of $\GL_2 \times \GL_3\times \GL_4$, we can work with the orbit representatives.
    
    The pencils associated to $T_{19},\dots,T_{22}$ are
\begin{align*}
\ell_{T_{19}}=\begin{pmatrix} 
u & v & 0 & 0 \\
0 & u & v & 0\\
0 & 0 & 0 & u
\end{pmatrix}&; \quad 
\ell_{T_{20}}=\begin{pmatrix} 
u & v & 0 & 0 \\
0 & 0 & u & 0\\
0 & 0 & 0 & u
\end{pmatrix}; \quad 
\ell_{T_{21}}=\begin{pmatrix} 
u & v & 0 & 0 \\
0 & 0 & u & v \\
0 & 0 & 0 & u
\end{pmatrix}; \\ 
\ell_{T_{22}}=&\begin{pmatrix} 
u & v & 0 & 0 \\
0 & 0 & u & 0\\
0 & 0 & 0 & v
\end{pmatrix}; \quad 
\ell_{T_{23}}=\begin{pmatrix} 
u & v & 0 & 0 \\
0 & u & v & 0\\
0 & 0 & u & v
\end{pmatrix}
\end{align*}
respectively. The saturation of the ideals of $3$-minors are:
\begin{align*}
    {\rm minors}_3(T_{19})^{\rm sat} = (u), \quad {\rm minors}_3(T_{20})^{\rm sat} = {\rm minors}_3(T_{21})^{\rm sat} = (u^2), \\ {\rm minors}_3(T_{22})^{\rm sat} = (uv), \quad {\rm minors}_3(T_{23})^{\rm sat} = (1).
\end{align*}
In order to distinguish the cases n.\idref{20} and n.\idref{21}, note that the saturation of the ideals of $2$-minors are:
\[
    {\rm minors}_2(T_{20})^{\rm sat} = (u), \quad {\rm minors}_2(T_{21})^{\rm sat} = (1). \qedhere
\]
\end{proof}

\begin{algorithm}[htb]
\caption{Test if a rank-$1$ tensor lies in the decomposition locus of a rank-$5$ tensor in $\bbC^2 \otimes \bbC^3 \otimes \bbC^4$}
\flushleft{
\textbf{Input:} }
\begin{itemize}
    \item Rank-$5$ tensor $T \in \bbC^2 \otimes \bbC^3 \otimes \bbC^4$.
    \item Rank-one tensor $P \in \bbC^2 \otimes \bbC^3 \otimes \bbC^4$.
\end{itemize}
\flushleft{
\textbf{Output:} Whether $P \in \calF_T$ or $P \not\in \calF_T$.
}

\flushleft{
\textbf{Procedure:}
}
\begin{enumerate}
    \item[{\it Step 1.}] Compute the principal ideals of maximal minors of all flattenings $I_1 = (f_1), I_2 = (f_2), I_3 = (f_3)$ in $\bbC[\lambda]$. If $\deg(f_i) > 0$ for any $i$, then $T-\lambda P$ is not concise for any root of $f_i$; \\
    \textbf{return} $P \not\in \calF_{T_{21}}$.

    \item[\textit{Step 2.}] Compute the ideal maximal minors of the pencil $T-\lambda P$ in $\bbC[\lambda][u,v].$ 
    \begin{itemize}
        \item If it is equal to $(\ell(u,v)^2)$ for some linear form $\ell$ then it means that for any $\lambda$ the intersection of the pencil $T_{21}-\lambda P$ with $\sigma(X_{34})$ is a $2$-jet, \\
        \textbf{return} $P \in \calF_{T_{21}}$;
        \item else,  \textbf{return} $P \not\in \calF_{T_{21}}$.
    \end{itemize}
\end{enumerate}
\label{algo:234_rank5}
\end{algorithm}

Now, we are ready to compute symbolically the forbidden locus of $T_{21}$ and we will follow the procedure described in \Cref{algo:234_rank5}.
\begin{proposition} The forbidden locus of $T_{21}$ is
    \begin{align*}
        \calF_{T_{21}} = & \Big[(Z(b_3) \cap Z(\alpha')) \smallsetminus Z(c_1)\Big] \cap \Big[Z(b_1,b_2) \cup Z(b_2,a_1c_1+a_2c_2) \cup Z(a_2,b_1c_1+b_2c_3)\Big] \\ 
            & \cup \Big[(Z(b_3) \cap Z(\alpha') \cap Z(c_1)\Big] \cap \Big[Z(b_2,a_2b_1c_2) \cup Z(a_2,a_2b_2c_3) \cup [Z(c_3) \smallsetminus (Z(b_2) \smallsetminus Z(a_2b_1c_2))]\Big] \\
            & \cup [Z(a_2) \smallsetminus Z(b_3)] \cap Z(c_1,c_3),
    \end{align*}
    where 
    $
        \alpha' = a_{1}b_{1}c_{1}^{2}+a_{2}b_{1}c_{1}c_{2}+a_{1}b_{2}c_{1}c_{3}+a_{2}b_{2}c_{2}c_{3}$.
\end{proposition}
\begin{proof}
    Since $\calO_{21}$ is the only orbit of rank-$5$ tensors in $\bbC^2\otimes\bbC^3\otimes\bbC^4$, we have that $p \in \calF_{T_{21}}$ if and only if $T-\lambda p \in \calO_{21}$ for all $\lambda$'s. In order to check this condition we apply \Cref{lemma:234_rank5}.

    We first compute the ideals of maximal minors of all flattenings. In particular, we get
    \begin{align}\label{eq:orbit21_minors_flattenings}
        I = (1), \quad I_2 = \left(c_{3},\,c_{2},\,c_{1},\,a_{2},\lambda a_{1}b_{3}c_{4}-1\right), \nonumber \\ I_3 = \left(b_{3},\,b_{2},\lambda(a_{1}b_{1}c_{1}+a_{2}b_{1}c_{2})-1\right) \cap \left(b_{3},\,a_{2},\lambda(a_{1}b_{1}c_{1}+a_{1}b_{2}c_{3})-1\right).
    \end{align}

    Now, by \Cref{lemma:234_rank5}, we know that $T-\lambda P \in \calO_{21}$ if and only if it is concise and the corresponding pencil of $3\times 4$ matrices intersects the $2$nd secant variety of $X_{34}$ is a $2$-jet supported at a rank-$2$ matrix. Hence, we compute the ideal of maximal minors of the pencil corresponding to $T-\lambda P$ (this is saturated with respect to $(u,v) \cap (\lambda)$). That is the ideal in $\bbC[\lambda,\bfa,\bfb,\bfc][u,v]$ given by the intersection of the following primary components:
\begin{align*}
        \mathfrak{p}_1 & = \left(b_{3}^{2},\,b_{3}u,\,b_{2}u-b_{3}v,\,u^{2}\right), \\
        \mathfrak{p}_2 & = \left(a_{2}^{2},\,a_{2}u,\,a_{1}u+a_{2}v,\,u^{2}\right), \\
        \mathfrak{p}_3 & = \Big(c_{2}u-c_{1}v,~\alpha(\bfa,\bfb,\bfc)\lambda - c_1^2, \\ &  \left(\lambda(a_{1}b_{1}c_{1}^{2}+a_{1}b_{2}c_{1}c_{3}+a_{1}b_{3}c_{1}c_{4})-c_{1}\right)u+\left(\lambda(a_{2}b_{1}c_{1}^{2}+a_{2}b_{2}c_{1}c_{3}-a_{1}b_{3}c_{1}c_{3}-a_{2}b_{3}c_{2}c_{3}+a_{2}b_{3}c_{1}c_{4}\right)v, \\
        & \left(\lambda(a_{1}b_{1}c_{1}+a_{1}b_{2}c_{3}+a_{1}b_{3}c_{4})-1\right)u^{2}+\left(\lambda(a_{2}b_{1}c_{1}+a_{2}b_{2}c_{3}-a_{1}b_{3}c_{3}+a_{2}b_{3}c_{4}\right)uv-\lambda a_{2}b_{3}c_{3}v^{2}\Big),
    \end{align*}
where 
    \[
        \alpha(\bfa,\bfb,\bfc) := a_{1}b_{1}c_{1}^{3}+a_{2}b_{1}c_{1}^{2}c_{2}+a_{1}b_{2}c_{1}^{2}c_{3}+a_{2}b_{2}c_{1}c_{2}c_{3}-a_{1}b_{3}c_{1}c_{2}c_{3}-a_{2}b_{3}c_{2}^{2}c_{3}+a_{1}b_{3}c_{1}^{2}c_{4}+a_{2}b_{3}c_{1}c_{2}c_{4}.
    \]
    If $b_3a_2\cdot\alpha(\bfa,\bfb,\bfc) \neq 0$, then, for a generic $\lambda$, $T_{21}-\lambda P$ does not belong to $\sigma_2(X_{34})$ and, in particular, $P \not\in \calF_{T_{21}}$.
    
    We deduce that $\calF_{T_{21}} \subset Z(b_3) \cup Z(a_2) \cup Z(\alpha(\bfa,\bfb,\bfc)).$ 
    \begin{itemize}
        \item Let $P \in Z(\alpha(\bfa,\bfb,\bfc)) \smallsetminus [Z(b_3) \cup Z(a_2)]$. Computationally, we consider this case by passing to the quotient ring modulo $(\alpha(\bfa,\bfb,\bfc))$ and by saturating the computed ideal also by $(a_2b_3)$. Then, we get only two primary components given by 
        \[
          \left(c_{3}^{2},\,c_{2}c_{3}-c_{1}c_{4},\,c_{1}c_{3},\,c_{1}^{2},\,c_{4}u-c_{3}v,\,c_{3}u,\,c_{2}u-c_{1}v,\,c_{1}u,\,u^{2}\right)
        \]
        and 
  \begin{align*}
        & \Big(c_{2}^{2},\,c_{1}c_{2},\,c_{1}^{2},\,c_{2}u-c_{1}v,\\ & \left(\lambda(a_{1}b_{2}c_{1}c_{3}+a_{1}b_{3}c_{1}c_{4})-c_{1}\right)u+\lambda\left(a_{2}b_{2}c_{1}c_{3}-a_{1}b_{3}c_{1}c_{3}-a_{2}b_{3}c_{2}c_{3}+a_{2}b_{3}c_{1}c_{4}\right)v,
       \\
       & \left(\lambda(a_{1}b_{1}c_{1}+a_{1}b_{2}c_{3}+a_{1}b_{3}c_{4})-1\right)u^{2}+\lambda\left(a_{2}b_{1}c_{1}+a_{2}b_{2}c_{3}-a_{1}b_{3}c_{3}+a_{2}b_{3}c_{4}\right)uv-\lambda a_{2}b_{3}c_{3}v^{2}\Big).
        \end{align*}       
        It is immediate to notice that, if $c_1 \neq 0$, then we have empty intersection. If $c_1 = 0$ but $c_2 \neq 0$, then only the first component is not irrelevant and it is immediate to see that we always get a simple point. If $c_1 = c_2 = 0$, then we are left with the principal ideal generated by 
        \[
         \left(\lambda(a_{1}b_{2}c_{3}+a_{1}b_{3}c_{4})-1\right)u^{2}+\lambda\left(a_{2}b_{2}c_{3}-a_{1}b_{3}c_{3}+a_{2}b_{3}c_{4}\right)uv-\lambda a_{2}b_{3}c_{3}v^{2}.
        \]
        If $c_3 = 0$, then the above expression factors as $u((\lambda a_1b_3c_4-1)u - \lambda a_2b_3c_4v)$ 
        which gives us two distinct points (since $c_4$ cannot be also equal to zero). If $c_3 \neq 0$, then for a general $\lambda$ such principal generator is not a square and then we get again two distinct points. 
        
        In other words, we have that if $P \in Z(\alpha(\bfa,\bfb,\bfc)) \smallsetminus [Z(b_3) \cup Z(a_2)]$ then $T_{21}-\lambda P \not\in \calO_{21}$ for all $\lambda \neq 0$, i.e., $P \not\in \calF_{T_{21}}$.

        \item Let $P \in Z(b_3)$. Then the ideal of maximal minors of the pencil $T-\lambda P$ becomes 
       \begin{align*}
            & (u)~ \cap ~(b_2,u^2)~ \cap ~(a_2,u^2)  \\
            & \cap \Big(c_2u-c_1v, ~\lambda\alpha'-c_{1}, ~\left(\lambda(a_{1}b_{1}c_{1}+a_{1}b_{2}c_{3})-1\right)u+\lambda \left(a_{2}b_{1}c_{1}+a_{2}b_{2}c_{3}\right)v\Big)
        \end{align*} 
        with $\alpha'(\bfa,\bfb,\bfc) := a_{1}b_{1}c_{1}^{2}+a_{2}b_{1}c_{1}c_{2}+a_{1}b_{2}c_{1}c_{3}+a_{2}b_{2}c_{2}c_{3}$.

        If $\alpha' \neq 0$, then we get either two simple points (if $c_1 \neq 0$) or one simple point (if $c_1 = 0$). 

        If we additionally restrict to $\alpha' = 0$ and $c_1 \neq 0$, then the ideal of maximal minors of $T_{21}-\lambda P$ becomes 
\begin{align*} 
\left(u,\,b_{1}c_{1}+b_{2}c_{3}\right)
  \cap \left(u,\,a_{1}c_{1}+a_{2}c_{2}\right)
 \cap \left(b_{2},\,b_{1},\,u^{2}\right)  \\
 \cap \left(a_{2},\,b_{1}c_{1}+b_{2}c_{3},\,u^{2}\right)
  \cap \left(b_{2},\,a_{1}c_{1}+a_{2}c_{2},\,u^{2}\right) \cap (a_1,a_2,u^2). 
\end{align*}     
This defines a $2$-jet only if $a_2b_2 = 0$. In particular: if $a_2 = 0$, then it becomes $(a_1,u^2) \cap (b_1c_1+b_2c_3,u^2)$; if $b_2 = 0$, then it becomes $(b_1,u^2) \cap (a_1c_1+a_2c_2,u^2)$. 

If we restrict to $\alpha' = c_1 = 0$, then the ideal of maximal minors of $T_{21}-\lambda P$ becomes
\begin{align*}
(c_3,u^2) \cap (c_2,u) \cap (c_2,(\lambda a_1b_2c_3 -1)u + \lambda a_2b_2c_3) \cap (b_2,u^2) \cap (a_2,u^2) \\
\cap (b_2,u^3,\lambda a_2b_1c_2-1) \cap (a_2,u^3,\lambda a_1b_2c_3 -1).   
\end{align*}
        This defines a $2$-jet only if $a_2b_2c_3 = 0$. In particular: if $a_2 = 0$, then it becomes $(u^2) \cap (c_2,\lambda a_1b_2c_3 - 1) \cap (u^3,\lambda a_1b_2c_3 - 1)$;         if $b_2 = 0$, then it becomes $(u^2) \cap (u^3,\lambda a_2b_1c_2 - 1)$;   
        if $c_3 = 0$, then it becomes $(u^2) \cap (b_2,u^3,\lambda a_2b_1c_2 - 1)$. 
        Note that the cases in which the intersection becomes a $3$-jet correspond to the cases in which $T_{21}-\lambda P$ is not concise by \eqref{eq:orbit21_minors_flattenings}.

        Hence, $  \calF_{T_{21}} \cap Z(b_3)$ equals 
        \begin{align*}
           & \Big[(Z(b_3) \cap Z(\alpha')) \smallsetminus Z(c_1)\Big] \cap \Big[Z(b_1,b_2) \cup Z(b_2,a_1c_1+a_2c_2) \cup Z(a_2,b_1c_1+b_2c_3)\Big] \\ 
            & \cup \Big[(Z(b_3) \cap Z(\alpha') \cap Z(c_1)\Big] \cap \Big[Z(b_2,a_2b_1c_2) \cup Z(a_2,a_2b_2c_3) \cup [Z(c_3) \smallsetminus (Z(b_2) \smallsetminus Z(a_2b_1c_2))]\Big].
        \end{align*}

        \item Let $P \in Z(a_2) \smallsetminus Z(b_3)$. Computationally, we now work modulo $a_2$ and we saturate the ideal of maximal minors also by $a_1b_3$ since $\bfa \neq 0$ and $b_3 = 0$ has been already considered. Then, the ideal becomes
        \[
 (u) \cap \left(c_{2}u-c_{1}v,\lambda\alpha''-c_{1},\,\left(\lambda (a_{1}b_{1}c_{1}+
       a_{1}b_{2}c_{3}+a_{1}b_{3}c_{4})-1\right)u-\lambda a_{1}b_{3}c_{3}v\right),        
        \]
        with $\alpha''(\bfa,\bfb,\bfc) := a_{1}b_{1}c_{1}^{2}+a_{1}b_{2}c_{1}c
       _{3}-a_{1}b_{3}c_{2}c_{3}+a_{1}b_{3}c_{1}c_{4}$.

    If $\alpha'' \neq 0$, then we get either two simple points $(c_1 \neq 0)$ or one simple point $(c_1 = 0)$.

    If we additionally restrict to $\alpha'' = 0$ and $c_1 \neq 0$, then the ideal of maximal minors of $T_{21}-\lambda P$ becomes $(u,b_1c_1^2+b_2c_1c_3-b_3c_2b_3+b_3c_1c_4)$,     which never defines a $2$-jet.

    If instead we restrict to $\alpha'' = c_1 = 0$, then we get the ideal $(c_3,u^2) \cap (c_2,u) \cap (c_2,(\lambda(a_1b_2c_3+a_1b_3c_4)-1)u-\lambda a_1b_3c_3v)$.     This defines a $2$-jet if and only if $c_3 = 0$. Note that if $c_1 = c_3 = 0$ then $\alpha'' = 0$.

    Therefore, we conclude that
    \[
        \calF_{T_{21}} \cap [Z(a_2) \smallsetminus Z(b_3)] = [Z(a_2) \smallsetminus Z(b_3)] \cap Z(c_1,c_3). \qedhere
    \] 
        \end{itemize}
\end{proof}

 \subsection{Orbits n.\protect\idref{24} and n.\protect\idref{25}: concise tensors in $\bbC^2\otimes\bbC^3\otimes\bbC^5$}\label{ssec:n24_n25} Concise tensors in $\bbC^2\otimes\bbC^3\otimes\bbC^5$ have rank and border rank $5$ and are divided into two families:
 \begin{align*}
    T_{24}=&\bf{e}^1_1\otimes (\bf{e}^2_1\otimes \bf{e}^3_1+\bf{e}^2_2\otimes \bf{e}^3_3+\bf{e}^2_3\otimes \bf{e}^3_5) + \bf{e}^1_2 \otimes (\bf{e}^2_1\otimes \bf{e}^3_2+\bf{e}^2_2\otimes \bf{e}^3_4),\\
    T_{25}=&\bf{e}^1_1\otimes (\bf{e}^2_1\otimes \bf{e}^3_1+\bf{e}^2_2\otimes \bf{e}^3_2+\bf{e}^2_3\otimes \bf{e}^3_4) + \bf{e}^1_2 \otimes (\bf{e}^2_1\otimes \bf{e}^3_2+\bf{e}^2_2\otimes \bf{e}^3_3+ \bf{e}^2_3\otimes \bf{e}^3_5).
 \end{align*}
 We compute the forbidden loci of these tensors. By \Cref{lemma:234_rank5}, we have that, given $T \in \calO_{24} \cup \calO_{25}$, then $P \in \calF_{T}$ if and only, for any $\lambda$ either $T-\lambda P$ is concise or belongs to the orbit $\calO_{21}$.

 \begin{proposition} The forbidden locus of $T_{24}$ is
     \begin{align*}
        \calF_{T_{24}} = \bbC^2\otimes\bbC^3\otimes\bbC^5 \smallsetminus \Big[&[Z(a_2,c_1,c_2,c_3,c_4) \smallsetminus Z(a_1b_3c_5)] \cup [Z(a_2b_3) \smallsetminus V(T_{24}*)] \Big]. 
     \end{align*}
 \end{proposition}
 \begin{proof}
     Let $P = \bfa \otimes \bfb \otimes \bfc$. Then, the maximal minors of all flattenings of $T_{24} - \lambda P$ are 
     \begin{align*}
        & I_1 = (1), \quad I_2 = \left(c_{4},\,c_{3},\,c_{2},\,c_{1},\,a_{2},\,\lambda a_{1}b_{3}c_{5}-1\right), \\
        &I_3 = \left(a_{2}b_{3}, \lambda(a_{1}b_{1}c_{1}+a_{2}b_{1}c_{2}+a_{1}b_{2}c
        _{3}+a_{2}b_{2}c_{4}+a_{1}b_{3}c_{5})-1\right).
     \end{align*}
     Assume that we are in a case in which the third flattening has not maximal rank, i.e., computationally we work over the quotient ring modulo $I_3$. The ideal of $3 \times 3$ minors of the pencil of $3\times 5$ matrices $T_{24} - \lambda P$ is 
     \[
\left(b_{3},\,u\right) \cap \left(a_{2},\,u\right) \cap \left(b_{3},\,c_{2}c_{3}-c_{1}c_{4},\,c_{4}u-c_{3}v,\,c_{2}u-c_{1}v\right) \cap \left(a_{2},\,c_{2}c_{3}-c_{1}c_{4
        },\,c_{4}u-c_{3}v,\,c_{2}u-c_{1}v\right).     
     \]
     Clearly, there is not choice of $P$ such that this intersection is a $2$-jet, i.e., such that $T_{24}- \lambda P \in \calO_{21}$. Therefore, $\calF_{T_{24}}$ consists of all rank-one tensors $P$ such that $T-\lambda P$ is concise for any $\lambda$.
 \end{proof}

  \begin{proposition}The forbidden locus of $T_{25}$ is
     \[
        \calF_{T_{25}} = \bbC^2\otimes\bbC^3\otimes\bbC^5 \smallsetminus \Big[\big[Z(a_2b_1-a_1b_2) \smallsetminus Z(a_{1}b_{1}c_{1}+a_{1}b_{2}c_{2}+a_{2}b_{2}c_{3}+a_{1}b_{3}c_{4}+a_{2}b_{3}c_{5})\big] \smallsetminus Z(c_4,c_5,c_2^2-4c_1c_3)\Big].
     \]
 \end{proposition}
 \begin{proof}
     Let $P = \bfa \otimes \bfb \otimes \bfc$. Then, the maximal minors of all flattenings of $T_{25} - \lambda P$ are 
     \begin{align*}
        & I_1 = I_2 = (1), \quad I_3 = \left(a_{2}b_{1}-a_{1}b_{2}, \lambda(a_{1}b_{1}c_{1}+a_{1}b_{2}c_{2}+a_{2}b_{2}c_{3}+a_{1}b_{3}c_{4}+a_{2}b_{3}c_{5})-1\right).
     \end{align*}
     Assume that we are in a case in which the third flattening has not maximal rank, i.e., computationally we work over the quotient ring modulo $I_3$. The ideal of $3 \times 3$ minors of the pencil of $3\times 5$ matrices $T_{25} - \lambda P$ is 
     \begin{equation}\label{eq:orbit25_3minors}
 \left(c_{5}u-c_{4}v,\,c_{3}c_{4}^{2}-c_{2}c_{4}c_{5}+c_{1}c_{5
        }^{2},\,c_{3}c_{4}u+\left(-c_{2}c_{4}+c_{1}c_{5}\right)v,\,c_{3}u^{2}-c_{2}uv+c_{1}v^{2}\right).    
     \end{equation}
     If $c_4 = c_5 = 0$, then the latter simplifies to $(c_{3}u^{2}-c_{2}uv+c_{1}v^{2})$ 
     which defines a $2$-jet if $c_2^2-4c_1c_3 = 0$. 

     If $c_4 \neq 0$ or $c_5 \neq 0$, then \eqref{eq:orbit25_3minors} cannot define a $2$-jet and, therefore, $T-\lambda P \not\in \calO_{21}$; in particular, $P \not\in \calF_{T_{25}}$.
 \end{proof}
 
 \subsection{Orbit n.\protect\idref{26}: tensors of maximal rank in $\bbC^2\otimes\bbC^3\otimes\bbC^6$}\label{ssec:n26}
 The maximal rank of a tensor $T\in \mathbb{C}^2\otimes \mathbb{C}^3\otimes\mathbb{C}^{n}$ is $6$ and it is achieved for $n=6$. The corresponding normal form for this class is 
$$
T_{26}=\bf{e}^1_1\otimes (\bf{e}^2_1\otimes \bf{e}^3_1+\bf{e}^2_2\otimes \bf{e}^3_3+\bf{e}^2_3\otimes \bf{e}^3_5)+\bf{e}^1_2\otimes (\bf{e}^2_1\otimes \bf{e}^3_2+\bf{e}^2_2\otimes \bf{e}^3_4+\bf{e}^2_3\otimes \bf{e}^3_6).
$$
Therefore, a rank-one tensor $P = \bfa\otimes\bfb\otimes\bfc$ belongs to the forbidden locus of $T_{26}$ if and only if $T_{26}-\lambda P$ is concise for all $\lambda$'s. 
\begin{proposition}
{ Let $T \in \calO_{26} \subset \bbC^2 \otimes \bbC^3 \otimes \bbC^6$ such that $(A, B, C) \cdot T = T_{26}$. Then, $$\calF_T = V((A^T, B^T, C^T) \cdot T^*).$$} 
\end{proposition}
\begin{proof}
Let $P = \bfa\otimes\bfb\otimes\bfc \in \bbC^2\otimes \mathbb{C}^3\otimes\mathbb{C}^6$. The maximal minors of all three flattenings of $T_{26}-\lambda P$ are \[I_1 = I_2 = (1), I_3 = (\lambda \langle T_{26}^*,P \rangle -1).\]
 Now, let $T \in \calO_{26}$ such that $(A \otimes B \otimes C) \cdot T = T_{26}$. Then, {as observed in \Cref{rmk:change_coordinates}}, $\bfa \otimes \bfb \otimes \bfc \in \calF_T$ if and only if $(A \otimes B \otimes C) \cdot (\bfa \otimes \bfb \otimes \bfc) \in \calF_{T_{26}}$.
This concludes the proof.
\end{proof}

\bibliographystyle{alpha}
\bibliography{References_BOS.bib}

\newcommand{\etalchar}[1]{$^{#1}$}
\begin{thebibliography}{SDLF{\etalchar{+}}17}

\bibitem[AH97]{AH1997jets}
J.~Alexander and A.~Hirschowitz.
\newblock Interpolation on jets.
\newblock {\em Journal of Algebra}, 192(1):412--417, 1997.

\bibitem[BB13]{BB13}
E.~Ballico and A.~Bernardi.
\newblock Tensor ranks on tangent developable of segre varieties.
\newblock {\em Linear and Multilinear Algebra}, 61(7):881 -- 894, 2013.

\bibitem[BB17]{BB17}
E.~Ballico and A.~Bernardi.
\newblock A uniqueness result on the decompositions of a bi-homogeneous
  polynomial.
\newblock {\em Linear Multilinear Algebra}, 65(4):677--698, 2017.

\bibitem[BBC{\etalchar{+}}19]{QuantumLectures}
E.~Ballico, A.~Bernardi, I.~Carusotto, S.~Mazzucchi, and V.~Moretti, editors.
\newblock {\em Quantum physics and geometry}, volume~25 of {\em Lecture Notes
  of the Unione Matematica Italiana}.
\newblock Springer, Cham, 2019.

\bibitem[BBCG19]{ballico2019partially}
E.~Ballico, A.~Bernardi, M.~Christandl, and F.~Gesmundo.
\newblock On the partially symmetric rank of tensor products of $ w $-states
  and other symmetric tensors.
\newblock {\em Rendiconti Lincei}, 30(1):93--124, 2019.

\bibitem[BBS21]{BBS}
E.~Ballico, A.~Bernardi, and P.~Santarsiero.
\newblock Identifiability of rank-3 tensors.
\newblock {\em Mediterr. J. Math.}, 18(4):Paper No. 174, 26, 2021.

\bibitem[BBS23]{BBS:appendix}
E.~Ballico, A.~Bernardi, and P.~Santarsiero.
\newblock Appendix in an algorithm for the non-identifiability of rank-3
  tensors.
\newblock {\em Bollettino dell'Unione Matematica Italiana}, pages 1--30, 2023.

\bibitem[BCC{\etalchar{+}}18]{bernardi2018hitchhiker}
A.~Bernardi, E.~Carlini, M.~V. Catalisano, A.~Gimigliano, and A.~Oneto.
\newblock The hitchhiker guide to: Secant varieties and tensor decomposition.
\newblock {\em Mathematics}, 6(12):314, 2018.

\bibitem[BCGI09]{Bernardi2009982}
A.~Bernardi, M.V. Catalisano, A.~Gimigliano, and M.~Idà.
\newblock Secant varieties to osculating varieties of veronese embeddings of
  pn.
\newblock {\em Journal of Algebra}, 321(3):982 – 1004, 2009.

\bibitem[BL13]{BL}
J.~Buczy\'{n}ski and J.~M. Landsberg.
\newblock Ranks of tensors and a generalization of secant varieties.
\newblock {\em Linear Algebra Appl.}, 438(2):668--689, 2013.

\bibitem[BL14]{BL14}
J.~Buczy{\'n}ski and J.M. Landsberg.
\newblock On the third secant variety.
\newblock {\em J Algebr Comb}, 40:475 -- 502, 2014.

\bibitem[CCO17]{CCO}
E.~Carlini, M.~V. Catalisano, and A.~Oneto.
\newblock Waring loci and the {S}trassen conjecture.
\newblock {\em Adv. Math.}, 314:630--662, 2017.

\bibitem[FBH{\etalchar{+}}22]{fawzi2022discovering}
A.~Fawzi, M.~Balog, A.~Huang, T.~Hubert, B.~Romera-Paredes, M.~Barekatain,
  A.~Novikov, F.~J. R~Ruiz, J.~Schrittwieser, G.~Swirszcz, et~al.
\newblock Discovering faster matrix multiplication algorithms with
  reinforcement learning.
\newblock {\em Nature}, 610(7930):47--53, 2022.

\bibitem[Fla24]{flavi2024decompositions}
C.~Flavi.
\newblock Decompositions of powers of quadrics.
\newblock {\em in preparation}, 2024.

\bibitem[GKZ08]{GKZ}
I.~M. Gelfand, M.~M. Kapranov, and A.~V. Zelevinsky.
\newblock {\em Discriminants, resultants and multidimensional determinants}.
\newblock Modern Birkh\"{a}user Classics. Birkh\"{a}user Boston, Inc., Boston,
  MA, 2008.
\newblock Reprint of the 1994 edition.

\bibitem[Gri78]{grigoriev1978multiplicative}
D.~Y. Grigoriev.
\newblock Multiplicative complexity of a pair of bilinear forms and of the
  polynomial multiplication.
\newblock In {\em International Symposium on Mathematical Foundations of
  Computer Science}, pages 250--256. Springer, 1978.

\bibitem[GS]{M2}
D.~R. Grayson and M.~E. Stillman.
\newblock Macaulay2, a software system for research in algebraic geometry.
\newblock Available at \url{http://www.math.uiuc.edu/Macaulay2/}.

\bibitem[HL13]{hillar2013most}
C.~J. Hillar and L.-H. Lim.
\newblock Most tensor problems are np-hard.
\newblock {\em Journal of the ACM (JACM)}, 60(6):1--39, 2013.

\bibitem[HLT12]{HLT}
F.~Holweck, J.-G. Luque, and J.-Y. Thibon.
\newblock Geometric descriptions of entangled states by auxiliary varieties.
\newblock {\em J. Math. Phys.}, 53:102203, 2012.

\bibitem[HOOS19]{HOOS}
J.~D. Hauenstein, L.~Oeding, G.~Ottaviani, and A.~J. Sommese.
\newblock Homotopy techniques for tensor decomposition and perfect
  identifiability.
\newblock {\em Journal f{\"u}r die reine und angewandte Mathematik (Crelles
  Journal)}, 2019(753):1--22, 2019.

\bibitem[HT23]{horobet2023does}
E.~Horobet and E.~Teixeira Turatti.
\newblock When does subtracting a rank-one approximation decrease tensor rank?
\newblock {\em arXiv preprint arXiv:2303.14985}, 2023.

\bibitem[IK99]{iarrobino1999power}
A.~Iarrobino and V.~Kanev.
\newblock {\em Power sums, Gorenstein algebras, and determinantal loci}.
\newblock Springer Science \& Business Media, 1999.

\bibitem[IR08]{ionescu2008varieties}
P.~Ionescu and F.~Russo.
\newblock Varieties with quadratic entry locus, ii.
\newblock {\em Compositio Mathematica}, 144(4):949--962, 2008.

\bibitem[J{\'a}J79]{jaja1979optimal}
J.~J{\'a}J{\'a}.
\newblock Optimal evaluation of pairs of bilinear forms.
\newblock {\em SIAM Journal on Computing}, 8(3):443--462, 1979.

\bibitem[Kac80]{K}
V.~G. Kac.
\newblock Some remarks on nilpotent orbits.
\newblock {\em J. Algebra}, 64(1):190--213, 1980.

\bibitem[KB09]{kolda2009tensor}
T.~G. Kolda and B.~W. Bader.
\newblock Tensor decompositions and applications.
\newblock {\em SIAM review}, 51(3):455--500, 2009.

\bibitem[Kru77]{kruskal}
J.~B. Kruskal.
\newblock Three-way arrays: rank and uniqueness of trilinear decompositions,
  with application to arithmetic complexity and statistics.
\newblock {\em Linear algebra and its applications}, 18(2):95--138, 1977.

\bibitem[Lan12]{L}
J.~M. Landsberg.
\newblock {\em Tensors: geometry and applications}, volume 128 of {\em Graduate
  Studies in Mathematics}.
\newblock American Mathematical Society, Providence, RI, 2012.

\bibitem[Lan17]{landsberg2017geometry}
J.~M. Landsberg.
\newblock {\em Geometry and complexity theory}, volume 169.
\newblock Cambridge University Press, 2017.

\bibitem[MO20]{MO}
B.~Mourrain and A.~Oneto.
\newblock On minimal decompositions of low rank symmetric tensors.
\newblock {\em Linear Algebra Appl.}, 607:347--377, 2020.

\bibitem[Ott13]{Ottaviani2013}
G.~Ottaviani.
\newblock {\em Introduction to the Hyperdeterminant and to the Rank of
  Multidimensional Matrices}, pages 609--638.
\newblock Springer New York, New York, NY, 2013.

\bibitem[Par01]{parfenov2001orbits}
P.~G. Parfenov.
\newblock Orbits and their closures in the spaces.
\newblock {\em Sbornik: Mathematics}, 192(1):89, 2001.

\bibitem[Rus09]{russo2009varieties}
F.~Russo.
\newblock Varieties with quadratic entry locus, i.
\newblock {\em Mathematische Annalen}, 344(3):597--617, 2009.

\bibitem[San23]{santarsiero2023algorithm}
P.~Santarsiero.
\newblock An algorithm for the non-identifiability of rank-3 tensors.
\newblock {\em Bollettino dell'Unione Matematica Italiana}, pages 1--30, 2023.

\bibitem[SDLF{\etalchar{+}}17]{sidiropoulos2017tensor}
N.~D. Sidiropoulos, L.~De~Lathauwer, X.~Fu, K.~Huang, E.~E. Papalexakis, and
  C.~Faloutsos.
\newblock Tensor decomposition for signal processing and machine learning.
\newblock {\em IEEE Transactions on signal processing}, 65(13):3551--3582,
  2017.

\bibitem[Tei86]{teichert1986komplexitat}
L.~Teichert.
\newblock {\em Die Komplexit{\"a}t von Bilinearformpaaren {\"u}ber beliebigen
  K{\"o}rpern}.
\newblock PhD thesis, PhD thesis, Mathematisch-Naturwissenschaftliche Fakultät
  der Technischen Universität Clausthal, 1986.

\bibitem[Ven19]{ventura2019real}
E.~Ventura.
\newblock Real rank boundaries and loci of forms.
\newblock {\em Linear and Multilinear Algebra}, 67(7):1404--1419, 2019.

\end{thebibliography}

\end{document}